\documentclass[11pt]{amsart} 
\usepackage{amsmath} \usepackage{amsxtra}
\usepackage{amscd} \usepackage{amsthm}  
\usepackage{amsfonts}\usepackage{amssymb} 
\usepackage{eucal}\usepackage{bm}
\usepackage{graphicx} 
\usepackage[latin9]{inputenc}  

\newtheorem{Cor}[subsubsection]{Corollary}
\newtheorem{Lm}[subsubsection]{Lemma}

\newtheorem{Pp}[subsubsection]{Proposition}
\newtheorem{Con}[subsubsection]{Conjecture}
\newtheorem{Thm}[subsubsection]{Theorem}
\newtheorem{Def}[subsubsection]{Definition}
\newtheorem{Rem}[subsubsection]{Remark}

\theoremstyle{definition}

\theoremstyle{remark}

\newcommand{\nc}{\newcommand}
\nc{\renc}{\renewcommand}
\nc{\ssec}{\subsection}
\nc{\sssec}{\subsubsection}
\nc{\on}{\operatorname}

\nc\ol{\overline}
\nc\wt{\widetilde}
\nc\tboxtimes{\wt{\boxtimes}}

\newcommand{\cA}{{\mathcal A}}
\newcommand{\cB}{{\mathcal B}}
\newcommand{\cC}{{\mathcal C}}
\newcommand{\cD}{{\mathcal D}}
\newcommand{\cH}{{\mathcal H}}
\newcommand{\cE}{{\mathcal E}}

\newcommand{\cI}{{\mathcal I}}

\newcommand{\cO}{{\mathcal O}}
\newcommand{\cL}{{\mathcal L}}
\newcommand{\cM}{{\mathcal M}}

\newcommand{\cF}{{\mathcal F}}
\newcommand{\cK}{{\mathcal K}}
\newcommand{\cP}{{\mathcal P}}

\newcommand{\cS}{{\mathcal S}}

\newcommand{\cV}{{\mathcal V}}
\newcommand{\cW}{{\mathcal W}}

\newcommand{\cY}{{\mathcal Y}}
\newcommand{\cZ}{{\mathcal Z}}

\renewcommand{\AA}{{\mathbb A}}

\newcommand{\GG}{{\mathbb G}}

\newcommand{\ZZ}{{\mathbb Z}}

\newcommand{\PP}{{\mathbb P}}

\newcommand{\HH}{{\mathbb H}}

\nc{\gi}{\mathfrak{i}}
\nc{\ga}{\mathfrak{a}}
\newcommand{\gm}{\mathfrak{m}}    

\newcommand{\gp}{\mathfrak{p}}
\newcommand{\gq}{\mathfrak{q}}

\newcommand{\gr}{\mathfrak{r}}

\nc{\gM}{\mathfrak{M}}
\nc{\gV}{\mathfrak{V}}
\nc{\gE}{\mathfrak{E}}

\nc{\bA}{\mathbf{A}}
\nc{\bC}{\mathbf{C}}

\nc{\uZ}{\underline{\cZ}}

\nc{\MAPS}{{\mathcal Maps}}

\newcommand{\Rep}{{\on{Rep}}}

\newcommand{\Qlb}{\mathbb{\bar Q}_\ell}
\newcommand{\Gm}{\mathbb{G}_m}

\newcommand{\Ql}{\mathbb{Q}_\ell}
\newcommand{\toup}[1]{\stackrel{#1}{\to}}

\newcommand{\hook}[1]{\stackrel{#1}{\hookrightarrow}}

\newcommand{\getsup}[1]{\stackrel{#1}{\gets}}

\newcommand{\Sp}{\on{\mathbb{S}p}}
\newcommand{\Spin}{\on{\mathbb{S}pin}}
\newcommand{\GSpin}{\on{G\mathbb{S}pin}}
\newcommand{\GSp}{\on{G\mathbb{S}p}}
\newcommand{\IC}{\on{IC}}

\newcommand{\Hom}{\on{Hom}}

\newcommand{\End}{\on{End}}

\newcommand{\Sym}{\on{Sym}}
\newcommand{\SO}{\on{S\mathbb{O}}}

\newcommand{\GSO}{\on{GS\mathbb{O}}}
\newcommand{\Ker}{\on{Ker}}

\newcommand{\Aut}{\on{Aut}}

\newcommand{\RG}{\on{R\Gamma}}

\newcommand{\Bun}{\on{Bun}}

\newcommand{\Bunt}{\on{\widetilde\Bun}}

\newcommand{\Spec}{\on{Spec}}

\newcommand{\supp}{\on{supp}}

\newcommand{\Gr}{\on{Gr}}

\newcommand{\Aff}{\on{Aff}}

\newcommand{\GL}{\on{GL}}

\newcommand{\pr}{\on{pr}}
\newcommand{\id}{\on{id}}

\newcommand{\QED}{$\square$} 
  
\newcommand{\iso}{{\widetilde\to}}

\newcommand{\comp}{\circ}

\renewcommand{\H}{{\on{H}}}   

\newcommand{\D}{\on{D}}       
\newcommand{\ov}[1]{\overline{#1}}
\newcommand{\select}[1]{{\it{#1}}}

\newcommand{\und}[1]{\underline{#1}}

\newcommand{\<}{\langle}

\newcommand{\Loc}{\on{Loc}}

\newcommand{\Sph}{\on{Sph}}
\newcommand{\Res}{\on{Res}}

\newcommand{\SL}{\on{SL}}

\newcommand{\Vect}{\on{Vect}}

\newcommand{\Ind}{\on{Ind}}

\newcommand{\LocSys}{\on{LocSys}}

\newcommand{\ra}{\rightarrow}
\newcommand{\la}{\leftarrow}

\nc{\Perv}{\on{Perv}}

\nc{\Gra}{\on{Gra}}
\nc{\PPerv}{\on{{\PP}erv}}

\nc{\oX}{\overset{\scriptscriptstyle\circ}{X}}
\nc{\ocL}{\overset{\scriptscriptstyle\circ}{\cL}}
\nc{\gRes}{\on{gRes}}
\nc{\Sign}{\on{Sign}}
\nc{\goodat}{\rm{good\, at}}
\nc{\Whit}{\on{Whit}}
\nc{\add}{\on{add}}
\nc{\FS}{\on{FS}}
\nc{\oo}[1]{\overset{\scriptscriptstyle\circ}{#1}}
\nc{\can}{\on{can}}
\nc{\summ}{\on{sum}}
\nc{\SiSu}{\on{SS}}
\nc{\Irr}{\on{Irr}}
\nc{\Hecke}{\on{Hecke}}
\nc{\oHecke}{\overset{\scriptstyle\bullet}{\Hecke}}
\nc{\og}[1]{\overset{\scriptscriptstyle\bullet}{#1}}
\nc{\of}{\overset{\scriptstyle\bullet}{f}}
\nc{\Exp}{\on{{\mathcal E}xp}}
\nc{\Chain}{\on{Chain}}
\nc{\Map}{\on{Map}}
\nc{\cSet}{\on{{\mathcal S}et}}
\nc{\Cat}{\on{\mathcal{C}at}}
\nc{\bfitDelta}{\bm{\mathit{\Delta}}}
\nc{\Grpd}{\on{Grpd}}
\nc{\Kan}{\on{{\mathcal K}an}}
\nc{\Spc}{\on{Spc}}
\nc{\Yon}{\on{Yon}}
\nc{\colim}{\on{colim}}
\nc{\Fin}{\on{{\mathcal F}in}}
\nc{\Alg}{\on{Alg}}
\nc{\Triv}{\on{\mathcal Triv}}
\nc{\Grp}{\on{{\mathcal G}rp}}
\nc{\EM}{\on{{\mathcal EM}}}
\nc{\Surj}{\on{Surj}}
\nc{\Ass}{\on{{\mathcal A}ss}}
\nc{\Sptr}{\on{Sptr}}
\nc{\cPr}{\on{{\cP}r}}
\nc{\Grd}{\on{Grd}}
\nc{\CAT}{\on{\bf 1-Cat}}
\nc{\DGCat}{\on{DGCat}}
\nc{\Act}{\on{Act}}
\nc{\Env}{\on{Env}}
\nc{\Quad}{\on{Quad}}
\nc{\ComGrp}{\on{ComGrp}}
\nc{\PreStk}{\on{PreStk}}
\nc{\Stk}{\on{Stk}}
\nc{\NearStk}{\on{NearStk}}
\nc{\Tot}{\on{Tot}}
\nc{\Ptd}{\on{Ptd}}
\nc{\Mon}{\on{Mon}}
\nc{\Idem}{\on{Idem}}
\nc{\ind}{\on{ind}}
\nc{\BMod}{\on{BMod}}
\nc{\Tens}{\on{Tens}}
\nc{\Step}{\on{Step}}
\nc{\MAP}{\on{\bf Map}}
\nc{\Seq}{\on{Seq}}
\nc{\boneCat}{\on{\mathbf{1-Cat}}}
\nc{\DG}{\on{DG}}
\nc{\WC}{\on{WC}}
\nc{\QCoh}{\on{QCoh}}
\nc{\Nilp}{\on{Nilp}}

\makeatletter
\newcommand*{\doublerightarrow}[2]{\mathrel{
  \settowidth{\@tempdima}{$\scriptstyle#1$}
  \settowidth{\@tempdimb}{$\scriptstyle#2$}
  \ifdim\@tempdimb>\@tempdima \@tempdima=\@tempdimb\fi
  \mathop{\vcenter{
    \offinterlineskip\ialign{\hbox to\dimexpr\@tempdima+1em{##}\cr
    \rightarrowfill\cr\noalign{\kern.5ex}
    \rightarrowfill\cr}}}\limits^{\!#1}_{\!#2}}}
\newcommand*{\triplerightarrow}[1]{\mathrel{
  \settowidth{\@tempdima}{$\scriptstyle#1$}
  \mathop{\vcenter{
    \offinterlineskip\ialign{\hbox to\dimexpr\@tempdima+1em{##}\cr
    \rightarrowfill\cr\noalign{\kern.5ex}
    \rightarrowfill\cr\noalign{\kern.5ex}
    \rightarrowfill\cr}}}\limits^{\!#1}}}
\makeatother

\begin{document}

\title{On the automorphic sheaves for $\GSp_4$}
\author{S. Lysenko}
\address{Institut Elie Cartan Lorraine, Universit\'e de Lorraine, 
 B.P. 239, F-54506 Vandoeuvre-l\`es-Nancy Cedex, France}
\email{Sergey.Lysenko@univ-lorraine.fr}
\begin{abstract}
In this paper we first review the setting for the geometric Langlands functoriality and establish a result for the `backward' functoriality functor. We illustrate this by known examples of the geometric theta-lifting. We then apply the above result to obtain new Hecke eigen-sheaves. The most important application is a construction of the automorphic sheaf for $G=\GSp_4$ attached to
a $\check{G}$-local system on a curve $X$ such that its standard representation is an irreducible local system of rank 4 on $X$. \end{abstract}
\maketitle

\section{Introduction}

\ssec{} Let $X$ be a smooth projective curve over an algebraically closed field $k$ of characteristic $p\ge 0$. The purpose of this paper is twofold: 
\begin{itemize}
\item[i)] We formulate a setting for the geometric Langlands functoriality (at the nonramified level) and prove an (easy) result on the `backward' functoriality functor. We illustrate our setting with known examples of the geometric theta-lifting and Eisenstein series.
\item[ii)] We apply this result combined with previously established results on the geometric theta-lifting to obtain new examples of automorphic sheaves. The most important application is a construction of new automorphic sheaves for $\GSp_4$. 
\end{itemize}

\sssec{} Let $G,H$ be split connected reductive groups over $k$, write $\check{G}, \check{H}$ for the Langlands dual groups over $k$. Write $\Bun_G$ for the stack of $G$-torsors on $X$, $\D(\Bun_G)$ for the derived category of $\Qlb$-sheaves on $\Bun_G$. Given a homomorphism $\bar\kappa: \check{G}\times\SL_2\to\check{H}$, let $\kappa: \check{G}\times\Gm\to \check{H}$ be its restriction to the standard maximal torus of $\SL_2$. According to the Langlands philosophy, one may ask about the corresponding geometric Langlands functoriality functor $F_H: \D(\Bun_G)\to \D(\Bun_H)$ commuting with the actions of $\Rep(\check{H})$. It is understood that $\Rep(\check{H})$ acts on $\D(\Bun_G)$ via its restriction through $\kappa$. The functor $F_H$ is expected to be given by some kernel $\cM$ on $\Bun_G\times\Bun_H$. Here $\SL_2$ is the \select{$\SL_2$ of Arthur}. We discuss the corresponding setting in Section~\ref{Sect_functoriality}. 

 Consider the `backward' functor $F_G: \D(\Bun_H)\to\D(\Bun_G)$ given by the same kernel $\cM$. Our first result is Theorem~\ref{Thm_Hecke_propety_of_cM_implies} describing the natural relation between $F_G$ and the Hecke functors on $\D(\Bun_H)$ and $\D(\Bun_G)$. This naturally leads to a notion of a partial Hecke property of $K\in\D(\Bun_G)$ with respect to a given $\check{H}$-local system $E_{\check{H}}$ on $X$ (cf. Definition~\ref{Def_1.1.2}). We illustrate our setting of the geometric Langlands functoriality with known cases of the geometric theta-lifting and geometric Eisenstein series in Sections~\ref{Section_Theta-lifting_2.2}-\ref{Sect_Liftings of Hecke eigen-sheaves}. 

\sssec{} In Section~\ref{Sect_Extending_Hecke_3.1} we obtain some results related to a possible extension of the partial Hecke property of $K\in\D(\Bun_G)$ to the whole Hecke property, provided that the corresponding local system $E_{\check{G}}$ is a $\check{G}$-local system. This question can be thought of in terms of the seminal recent paper \cite{G5}, where a spectral decomposition of the derived category $\D(\Bun_G)$ is established over the stack $\LocSys^{restr}_{\check{G}}$ of \select{restricted} $\check{G}$-local systems.  

 We establish Proposition~\ref{Pp_2.2.3}, which gives examples where a partial Hecke property can be extended to the whole Hecke property. Our proof of parts 3),4) of Proposition~\ref{Pp_2.2.3} uses (\cite{G5}, Theorem~10.5.2). It is used only to obtain some properties of our automorphic sheaves, not an additional structure. Our construction of the Hecke eigen-sheaves under consideration is algorithmic. 
 
\sssec{} We apply the above to construct new automorphic sheaves for $G=\GSp_4$. Namely, assume $k$ of characteristic $p>2$ (this is needed as we apply the results on the geometric theta-lifting). Let $E_{\check{G}}$ be a $\check{G}$-local system on $X$ such that its standard representation is an irreducible (rank 4) local system on $X$. 

 We construct an object $\cK$ of the derived category $\D(\Bun_G)$ of $\Qlb$-sheaves on $\Bun_G$, which is a $E_{\check{G}}$-Hecke eigen-sheaf. It is obtained via the geometric theta-lifting and confirms a conjecture proposed twelve years ago (\cite{L4}, Conjecture 6(ii)). The non-vanishing of $\cK$ is established in \cite{L5}. In the case when both the curve and $E_{\check{G}}$ are defined over a finite subfield $k_0\subset k$, we also give also an alternative argument showing that $\cK$ does not vanish. The complex $\cK$ is bounded from above over each open substack of $\Bun_G$ of finite type, and all its perverse cohomology sheaves are constructible.
 
  In Appendix~\ref{Appendix_B} we discuss some general properties of abelian categories over stacks related to the above question of the extending the partial Hecke property. 
  
\sssec{Organization} In Section~\ref{Sect_functoriality} we propose a setting for the geometric Langlands functoriality and prove Theorem~\ref{Thm_Hecke_propety_of_cM_implies}. We also illustrate our setting with known examples of the geometric theta-lifting and geometric Eisenstein series. In Section~\ref{Sect_Extending the Hecke property} we formulate our results that have not appeared in Section~\ref{Sect_functoriality}. They are related to the extension of a partial Hecke property and the construction of new automorphic sheaves. The proofs are collected in Section~\ref{Sect_Proofs}.

\ssec{Notation} 
\label{sect_Notation}
\sssec{} Work over an algebraically closed field $k$ of characteristic $p>2$. The  case $p=2$ is excluded as we are using the results of \cite{L1, L2, L3, LL3} on the the geometric Weil representation (they could possibly be extended to the $p=2$ case using \cite{GeL1}). All our stacks are defined over $k$. Let $\ell\ne p$ be a prime, $\Qlb$ the algebraic closure of $\Ql$. 
 
  We work with etale $\Qlb$-sheaves on algebraic stacks locally of finite type (cf. \cite{LO}). We ignore the Tate twists everywhere (they are easy to recover if necessary). 
Let $X$ be a smooth projective connected curve, $\Omega$ the canonical line bundle on $X$. 

 We use the following conventions from (\cite{L1}, Section~2.1). For an algebraic stack locally of finite type $S$ we denote by $\D(S)$ the derived category of unbounded $\Qlb$-complexes on $S$ with constructible cohomologies defined in (\cite{LO}, Remark~3.21) and denoted $\D_c(S,\Qlb)$ in \select{loc.cit.} For $*=+,-,b$ we have the corresponding full subcategory $\D^*(S)\subset \D(S)$ denoted $\D^*_c(S,\Qlb)$ in \select{loc.cit}. 
 
 Write $\D^-(S)_!\subset D^-(S)$ for the full subcategory of objects, which are extensions by zero from some open substack of finite type. Write $\D^{\prec}(S)\subset \D(S)$ for the full subcategory of complexes $K\in \D(S)$ such that for any open substack of finite type $U\subset S$, $K\mid_U\in \D^-(U)$. 
  
  For an algebraic stack locally of finite type $S$ we also consider the version $\D_{nc}(S)$ of $\D(S)$, where we do not assume that the cohomologies are constructible. It is defined as the homotopy category of the corresponding $\DG$-category of $\Qlb$-sheaves in \'etale topology on $S$, cf. \cite{GL, G4, G5, GR}. Then $\D(S)\subset\D_{nc}(S)$ is a full subcategory. For any morphism $f: S\to S'$ of algebraic stacks locally of finite type the direct image with compact support $f_!: \D_{nc}(S)\to \D_{nc}(S')$ is defined by passing to the homotopy for the corresponding functor between the $\DG$-categories  (\cite{G4}, Corollary~1.4.2). Similarly, we have the functors $f^!, f^*: \D_{nc}(S')\to \D_{nc}(S)$ and $f_*: \D_{nc}(S)\to \D_{ns}(S')$. 
  
  For an algebraic stack $S$ locally of finite type $\D_{nc}(S)$ is equipped with a perverse t-structure. The corresponding $\DG$-category of is left complete for this t-structure by (\cite{G5}, Th. 1.1.4), so by (\cite{HA}, 1.2.1.19)
\begin{itemize}
\item[(P)] If $K\in \D^{\prec}(S)$ is such that $\H^i(K)=0$ for all $i$ then $K=0$.
\end{itemize} 
We denoted by $\H^i(K)$ the $i$-cohomology sheaf of $K$ on $S$.

  In our proof of Proposition~\ref{Pp_2.2.3} we apply (\cite{G5}, Theorem~10.5.2), which requires the whole formalism of the $\DG$-categories of $\Qlb$-sheaves on prestacks from \cite{GL, G5}. However, we work mostly in the above setting of triangulated categories, not the $\DG$-setting of a theory of sheaves from \cite{G4,G5}. The reason is that the results of \cite{L1, L3} that we are using are for the moment established only on the level of triangulated categories.  
  
\sssec{} 
\label{Sect_def_Hecke_functors}
For a connected reductive group $G$ over $k$ write $\Bun_G$ for the stack of $G$-torsors on $X$. Denote by $\check{G}$ the Langlands dual group to $G$ over $\Qlb$. Denote by $\Rep(\check{G})$ the abelian category of finite-dimensional representations of $\check{G}$ over $\Qlb$. The trivial $G$-torsor on some base is denoted $\cF^0_G$. 

 For $x\in X$ let $\Gr_{G,x}$ denote the affine Grassmanian classifying $(\cF_G, \beta)$, where $\cF_G$ is a $G$-torsor on $X$, $\beta: \cF_G\,\iso\, \cF^0_G$ is a trivialization over $X-x$. Let $\cO_x$ be the completed local ring of $X$ at $x$, $D_x=\Spec\cO_x$. The spherical Hecke category $\Sph_G$ is defined as the category of $G(\cO_x)$-equivariant perverse sheaves on $\Gr_{G,x}$. The Satake equivalence provides an equivalence of symmetric monoidal categories $\Loc: \Rep(\check{G})\,\iso\, \Sph_G$.

 Write $\cH_G$ for the Hecke stack classifying $(x\in X, \cF_G\,\iso\,\cF'_G\mid_{X-x})$ with $\cF_G, \cF'_G\in\Bun_G$. It fits into the diagram
$$
\Bun_G\times X\getsup{h^{\la}_G\times\supp} \cH_G\toup{h^{\ra}_G}\Bun_G,
$$
where $h^{\la}_G$ (resp., $h^{\ra}_G$) sends the above point to $\cF_G$ (resp., $\cF'_G$). The map $\supp$ sends the above point to $x$. The Hecke functors 
\begin{equation}
\label{Hecke_functors_any_group}
\H^{\la}_G, \H^{\ra}_G: \Sph_G\times \D^{\prec}(S\times\Bun_G)\to \D^{\prec}(S\times\Bun_G\times X)
\end{equation}
are defined in (\cite{L1}, Section~2.1.1). Namely, for $x\in X$ write $_x\cH_G$ for the base change of $\cH_G$ by $\Spec k\toup{x} X$.
Write $\Bun_{G,x}$ for the stack classifying $\cF_G\in\Bun_G$ with a trivialization $\cF_G\,\iso\, \cF^0_G\mid_{D_x}$. Write $\id^l, \id^r$ for the isomorphisms
$$
_x\cH_G\,\iso\, \Bun_{G, x}\times^{G(\cO_x)}\Gr_{G,x}
$$
such that the projection on the first factor corresponds to $h^{\la}_G$, $h^{\ra}_G$ respectively. To $\cS\in\Sph_G$, $K\in\D^{\prec}(S\times\Bun_G)$ one attaches their twisted external products $(K\tboxtimes \cS)^l, (K\tboxtimes \cS)^r$ on $_x\cH_G$. They are normalized to be perverse for $K$ perverse. Then
$$
_x\H^{\la}_G, {_x\H^{\ra}_G}: \Sph_G\times \D^{\prec}(S\times\Bun_G)\to \D^{\prec}(S\times\Bun_G)
$$
are given by
$$
_x\H^{\la}_G(\cS, K)=(h^{\la}_G)_!(K\tboxtimes \ast\cS)^r\;\;\;\;\mbox{and}\;\;\;\;
{_x\H^{\ra}_G}(\cS, K)=(h^{\ra}_G)_!(K\tboxtimes \cS)^l
$$
We have denoted by $\ast: \Sph_G\,\iso\, \Sph_G$ the covariant equivalence induced by the map $G(F_x)\,\iso\,G(F_x), g\mapsto g^{-1}$. Letting $x$ vary in $X$, one similarly gets the functors (\ref{Hecke_functors_any_group}). 
Along the same line one also defines 
$$
\H^{\la}_G, \H^{\ra}_G: \Sph_G\times \D_{nc} (S\times\Bun_G)\to \D_{nc}(S\times\Bun_G\times X)
$$
In view of the Satake equivalence, we also view $\ast$ as a functor $\ast: \Rep(\check{G})\to\Rep(\check{G})$. It is induced by a Chevalley involution of $\check{G}$. 

\sssec{} For $\theta\in\pi_1(G)$ and $x\in X$ let $\Gr_G^{\theta}$ denote the connected component of $\Gr_G$ classifying $(\cF_G, \beta)$, where $\cF_G$ is a $G$-torsor on $X$, $\beta: \cF_G\,\iso\, \cF^0_G$ is a trivialization over $X-x$ such that $V_{\cF^0_G}\,\iso\, V_{\cF_G}(\<\theta, \check{\lambda}\>x)$ for a one dimensional $G$-module of  weight $\check{\lambda}$. Here $\pi_1(G)$ is the algebraic fundamental group of $G$ (the quotient of the coweights lattice by the coroots lattice). Write $\Bun_n$ for the stack of rank $n$ vector bundles on $X$. 
   
\sssec{} As in (\cite{L1}, Section~2.1.2), write $\Loc_X$ for the category of local systems on $X$ and set $\D\!\Loc_X=\oplus_{i\in\ZZ} \Loc_X[i]\subset \D(X)$. Then $\D\!\Loc_X$ as a symmetric monoidal category naturally. If $x\in X$ then a datum of a symmetric monoidal functor $E: \Rep(\check{G})\to \D\!\Loc_X$ is equivalent to a datum of a homomorphism $\sigma: \pi_1(X, x)\times\Gm\to \check{G}$ algebraic along $\Gm$ and continuous along $\pi_1(X,x)$. 

 As in (\cite{L1}, Section 2.1.2) we set 
$
\D\!\Sph_G=\oplus_{i\in\ZZ} \Sph_G[i]\subset \D(\Gr_G)
$ 
The above Satake equivalence is extended to an equivalence
of symmetric monoidal categories 
$$
\Loc^{\gr}: \Rep(\check{G}\times\Gm)\,\iso\, \D\Sph_G,
$$
so that $z\in\Gm$ acts on $V\in\Rep(\check{G}\times\Gm)$ by $z\mapsto z^{-r}$ if and only if $\Loc^{\gr}(V)\in \Sph_G[r]$. 
The above Hecke action is extended to an action of $\D\Sph_G$ on $\D^{\prec}(\Bun_G)$, so that $\H^{\la}_G(\cS[r],\cdot)=\H^{\la}_G(\cS, \cdot)[r]$ for $\cS\in\Sph_G$. We extend $\ast$ to a functor $\ast: \D\Sph_G\to\D\Sph_G$ by $\ast(\cS[r])\,\iso\, (\ast\cS)[r]$. 

\section{Geometric Langlands functoriality and the theta-lifting}
\label{Sect_functoriality}

In this section we formulate our results and definitions related to the geometric Langlands functoriality. We prove Theorem~\ref{Thm_Hecke_propety_of_cM_implies} consisting of two parts. The 'straight direction' part is known (cf. \cite{L1}, Theorem~3), and the `backward direction' part is new to the best of our knowledge. We review the geometric theta-lifting (cf. \cite{L1, Ga}) and the geometric Eisenstein series (cf. \cite{BG}) giving examples of the geometric Langlands functoriality. 
 
\ssec{Geometric Langlands functoriality} We use the following version of the notion of a Hecke eigen-sheaf from (\cite{G2}, Section~2.8). 
Write $s$ for the involution of $X^2$ permuting the two copies of $X$. The diagonal in $X^2$ is sometimes denoted $\triangle(X)$.

\begin{Def}
\label{Def_1.1.2}
(Partial Hecke property). Let $G,H$ be connected reductive groups over $k$, $\kappa: \check{G}\times\Gm\to\check{H}$ be a homomorphism over $\Qlb$. Pick $x\in X$. Let $E: \Rep(\check{H})\to \D\!\Loc_X$, $V\mapsto E^V$ be a symmetric monoidal functor. It yields the corresponding homomorphism $\sigma:\pi_1(X, x)\times\Gm\to\check{H}$. 

 We say that $K\in \D^{\prec}(\Bun_G)$ is equipped with a $E$-Hecke (or $\sigma$-Hecke) property with respect to $\kappa$ if for $V\in\Rep(\check{H})$ we are given a functorial isomorphism
\begin{equation}
\label{iso_for_Hecke_key}
\alpha_V: \H^{\la}_G(V, K)\,\iso\, K\boxtimes E^V[1]
\end{equation}
on $\Bun_G\times X$ such that H1) and H2) below are satisfied. First, for $V_1,\ldots, V_n\in\Rep(\check{H})$
iterating $\alpha$, one gets an isomorphism on $\Bun_G\times X^n$
$$
\alpha_{V_1,\ldots, V_n}:
\H^{\la}_G(V_1\boxtimes\ldots\boxtimes V_n, K)\,\iso\, K\boxtimes E^{V_1}[1]\boxtimes\ldots\boxtimes E^{V_n}[1]
$$
We require that for $V_1, V_2\in\Rep(\check{H})$ the following two diagrams commute\\
H1) 
$$
\begin{array}{ccc}
\H^{\la}_G(V_1\boxtimes V_2, K) & \toup{\alpha_{V_1, V_2}} & K\boxtimes E^{V_1}[1]\boxtimes E^{V_2}[1]\\
\downarrow &&\downarrow\\
(\id\times s)^*\H^{\la}_G(V_2\boxtimes V_1, K) & \toup{(\id\times s)^*\alpha_{V_2, V_1}} & K\boxtimes s^*(E^{V_2}[1]\boxtimes E^{V_1}[1]),\\
\end{array}
$$
H2) 
$$
\begin{array}{ccc}
\H^{\la}_G(V_1\boxtimes V_2, K)\mid_{\Bun_G\times \triangle(X)} & \toup{\alpha_{V_1, V_2}} & K\boxtimes E^{V_1}[1]\boxtimes E^{V_2}[1]\mid_{\Bun_G\times \triangle(X)} \\
\downarrow && \downarrow\\
\H^{\la}_G(V_1\otimes V_2, K)[1] & \toup{\alpha_{V_1\otimes V_2}} & K\boxtimes E^{V_1\otimes V_2}[2]
\end{array}
$$
Here the vertical arrows are the natural isomorphisms. It is understood that $\Rep(\check{H})$ acts on $\D^{\prec}(\Bun_G)$ via the restriction by $\kappa$.

 If moreover $G=H$ and $\kappa: \check{G}\times\Gm\to\check{G}$ is the projection then we say that $K$ is a $E$-Hecke eigen-sheaf. 
\end{Def}

 We insist that $K\in D^{\prec}(\Bun_G)$ in Definition~\ref{Def_1.1.2} has only a `partial Hecke property', the isomorphisms $\alpha_V$ are not given for all $V\in\Rep(\check{G})$, but only for a part of them.
\begin{Rem} The existence of the left vertical isomorphism in the diagram H1) of Definition~\ref{Def_1.1.2} is evident over $\Bun_G\times (X^2-\triangle(X))$. The fact that it extends to the whole of $\Bun_G\times X^2$ follows from (\cite{G2}, Proposition~2.8). 
\end{Rem} 
\begin{Rem} 
\label{Rem_Arthur_SL_2}
In the situation of Definition~\ref{Def_1.1.2} let $\sigma: \pi_1(X,x)\times\Gm\to \check{G}$ be a homomorphism, and $K\in\D^{\prec}(\Bun_G)$ be a nonzero $\sigma$-Hecke eigen-sheaf. According to the Arthur-Langlands philosophy, $\sigma$ should be a composition
$$
\pi_1(X,x)\times\Gm\toup{\id\times t} \pi_1(X,x)\times\SL_2\to \check{G},
$$ 
where $t: \Gm\hook{}\SL_2$ is the torus of diagonal matrices (\cite{F}, Section~4.3). 
\end{Rem} 
\begin{Rem} The notion of a Hecke eigen-sheaf is better formulated (from the point of view of the formalism of \cite{GR}) on a $\DG$-level,  for example as (\cite{GLys}, Section~9.5.3). Our Definition~\ref{Def_1.1.2} on the level of triangulated categories is a compromise as we will apply the results of \cite{L1, L3}, which are for the moment not established on the $\DG$-level.
\end{Rem}
 
\sssec{} 
\label{Sect_functoriality_first}
Let $G,H$ be connected reductive groups over $k$. 
Set for brevity $\Bun_{G,H}=\Bun_G\times\Bun_H$.
Assume given a complex $\cM\in\D^{\prec}(\Bun_{G,H})$. Consider the diagram of projections
$$
\Bun_H\getsup{p_H}\Bun_{G,H}\toup{p_G} \Bun_G
$$
Define the functors $F_H: \D^-(\Bun_G)_!\to \D^{\prec}(\Bun_H)$ and $F_G: \D^-(\Bun_H)_!\to\D^{\prec}(\Bun_G)$ by
$$
F_H(K)=(p_{H})_!((p_G^*K)\otimes\cM)[-\dim\Bun_G]
\;\;\;\mbox{and}\;\;\; F_G(K)=(p_G)_!((p_H^*K)\otimes\cM)[-\dim\Bun_H]
$$
The finiteness assumptions on the corresponding derived categories are imposed in order for the definition to make sense in the formalism of \cite{LO}. 

Removing the constructibility assumptions, one gets the functors $F_H: \D_{nc}(\Bun_G)\to \D_{nc}(\Bun_H)$ and $F_G: \D_{nc}(\Bun_H)\to\D_{nc}(\Bun_G)$ defined by the same formulas. They extend the corresponding lifting functors for constructible sheaves.
 
 For a scheme of finite type $S$
by abuse of notations we still denote by 
$$
F_H: \D_{nc}(\Bun_G\times S)\to \D_{nc}(\Bun_H\times S)
$$ 
the functor defined as above with $\cM$ replaced by $\pr^*\cM$, where $\pr: \Bun_H\times\Bun_G\times S\to \Bun_H\times\Bun_G$ is the projection (and similarly for $F_G$).

 Let $\kappa: \check{G}\times\Gm\to \check{H}$ be a homomorphism. We assume that the composition 
$$
\Rep(\check{H})\toup{\ast}\Rep(\check{H})\toup{\Res^{\kappa}}\Rep(\check{G}\times\Gm)
$$ 
is isomorphic to the composition $\Rep(\check{H})\toup{\Res^{\kappa}}\Rep(\check{G}\times\Gm)\toup{\ast} \Rep(\check{G}\times\Gm)$ (cf. Remark~\ref{Rem_2.1.4} below). This holds in all the examples we are interested in. Denote by $\kappa_{ex}: \check{G}\times\Gm\to \check{H}\times\Gm$ the map $(\kappa, \pr)$, where $\pr:\check{G}\times\Gm\to\Gm$ is the projection.
 
\begin{Def} (Straight direction). Assume given isomorphisms on $\Bun_H\times X$
$$
\alpha_V: \H^{\la}_H(V, F_H(K))\,\iso\, F_H\H^{\la}_G(V, K) 
$$
functorial in $K\in \D_{nc}(\Bun_G)$ and $V\in\Rep(\check{H})$. It is understood that $\Rep(\check{H})$ acts on $\D_{nc}(\Bun_G)$ via the restriction through $\kappa$. We require the properties H'1) and H'2) below. First, for $V_1,\ldots, V_n\in \Rep(\check{H})$ iterating $\alpha_V$, one gets the isomorphism over $\Bun_H\times X^n$
$$
\alpha_{V_1,\ldots, V_n}: \H^{\la}_H(V_1\boxtimes\ldots\boxtimes V_n, F_H(K))\,\iso\, F_H\H^{\la}_G(V_1\boxtimes\ldots\boxtimes V_n, K)
$$
We require that for $V_1, V_2\in\Rep(\check{H})$ the following diagrams commute\\
H'1) 
$$
\begin{array}{ccc}
\H^{\la}_H(V_1\boxtimes V_2, F_H(K)) & \toup{\alpha_{V_1, V_2}} & F_H\H^{\la}_G(V_1\boxtimes V_2, K)\\
\downarrow && \downarrow\\
(\id\times s)^*\H^{\la}_H(V_2\boxtimes V_1, F_H(K)) & \toup{(\id\times s)^*\alpha_{V_2, V_1}} & F_H(\id\times s)^*\H^{\la}_G(V_2\boxtimes V_1, K),
\end{array}
$$
\\
H'2)
$$
\begin{array}{ccc}
\H^{\la}_H(V_1\boxtimes V_2, F_H(K))\mid_{\Bun_H\times \triangle(X)} & \toup{\alpha_{V_1, V_2}} & F_H\H^{\la}_G(V_1\boxtimes V_2, K)\mid_{\Bun_H\times \triangle(X)}
\\
\downarrow && \downarrow\\
\H^{\la}_H(V_1\otimes V_2, F_H(K))[1] & \toup{\alpha_{V_1\otimes V_2}} & F_H \H^{\la}_G(V_1\otimes V_2, K)[1]
\end{array}
$$
Here the vertical arrows are natural isomorphisms. In this case we say that $F_H: \D_{nc}(\Bun_G)\to \D_{nc}(\Bun_H)$ commutes with the Hecke actions along $\kappa$ (so, realizes the geometric Langlands functoriality for $\kappa$). 
\end{Def}
\begin{Rem} 
\label{Rem_2.1.4}
Assume $F_H:\D_{nc}(\Bun_G)\to \D_{nc}(\Bun_H)$ commutes with Hecke actions along $\kappa$.  According to Arthur-Langlands philosophy, $\kappa$ is expected always to be a composition
$$
\check{G}\times\Gm\toup{\id\times t} \check{G}\times\SL_2\to \check{H},
$$
where $t: \Gm\to\SL_2$ is the torus of diagonal matrices.
\end{Rem}

\begin{Def} 
\label{Def_2.1.6}(Backward direction). Assume we are in the situation of Section~\ref{Sect_functoriality_first}, so $\kappa: \check{G}\times\Gm\to \check{H}$, and $F_G$ is the lifting functor in the backward direction. Assume given isomorphisms on $\Bun_G\times X$
$$
\alpha_V: \H^{\la}_G(V, F_G(K))\,\iso\, F_G\H^{\la}_H(V, K)
$$
functorial in $K\in \D_{nc}(\Bun_H)$ and $V\in\Rep(\check{H})$. 
It is understood that $\Rep(\check{H})$ acts on $\D_{nc}(\Bun_G)$ via the restriction through $\kappa$. We require the properties H''1) and H''2) below. First, for $V_1,\ldots, V_n\in \Rep(\check{H})$ iterating $\alpha_V$, one gets the isomorphism over $\Bun_G\times X^n$
$$
\alpha_{V_1,\ldots, V_n}: \H^{\la}_G(V_1\boxtimes\ldots\boxtimes V_n, F_G(K))\,\iso\, F_G\H^{\la}_H(V_1\boxtimes\ldots\boxtimes V_n, K)
$$
We require that for $V_1, V_2\in\Rep(\check{H})$ the following diagrams commute\\
H''1) 
$$
\begin{array}{ccc}
\H^{\la}_G(V_1\boxtimes V_2, F_G(K)) & \toup{\alpha_{V_1, V_2}}  & F_G\H^{\la}_H(V_1\boxtimes V_2, K) \\
\downarrow && \downarrow\\
(\id\times s)^*\H^{\la}_G(V_2\boxtimes V_1, F_G(K)) & \toup{(\id\times s)^*\alpha_{V_2, V_1}} & F_G(\id\times s)^*\H^{\la}_H(V_2\boxtimes V_1, K)
\end{array}
$$
H''2)
$$
\begin{array}{ccc}
\H^{\la}_G(V_1\boxtimes V_2, F_G(K))\mid_{\Bun_G\times \triangle(X)} & \toup{\alpha_{V_1, V_2}}  & F_G\H^{\la}_H(V_1\boxtimes V_2, K)\mid_{\Bun_G\times \triangle(X)} \\
\downarrow && \downarrow\\
\H^{\la}_G(V_1\otimes V_2, F_G(K))[1] & \toup{\alpha_{V_1\otimes V_2}} & F_G\H^{\la}_H(V_1\otimes V_2, K)[1]
\end{array}
$$
Here the vertical arrows are natural isomorphisms. In this case we say that $F_G$ commutes with the Hecke actions backward $\kappa$. 
\end{Def}

\begin{Cor} 
\label{Cor_2.1.8_partial_Hecke}
i) Assume $F_H:\D_{nc}(\Bun_G)\to \D_{nc}(\Bun_H)$ commutes with Hecke functors along $\kappa$. Let $x\in X$, $\sigma: \pi_1(X,x)\times\Gm\to \check{G}$ be a homomorphism. Write 
$$
\sigma^{ex}: \pi_1(X,x)\times\Gm\to \check{G}\times\Gm
$$ 
for the extension of $\sigma$, whose second component is the projection on $\Gm$. Let $\sigma_H=\kappa\comp \sigma^{ex}$. If $K\in\D^-(\Bun_G)_!$ is a $\sigma$-Hecke eigen-sheaf then $F_H(K)\in\D^{\prec}(\Bun_H)$ is naturally a $\sigma_H$-Hecke eigen-sheaf.

\smallskip\noindent
ii) Assume $F_G: \D_{nc}(\Bun_H)\to\D_{nc}(\Bun_G)$ commutes with Hecke functors backword $\kappa$. Let $x\in X$, $\sigma: \pi_1(X,x)\times\Gm\to \check{H}$ be a homomorphism. If $K\in\D^-(\Bun_H)_!$ is a $\sigma$-Hecke eigen-sheaf then $F_G(K)\in\D^{\prec}(\Bun_G)$ is naturally equipped with a $\sigma$-Hecke property with respect to $\kappa$. So, the Hecke property of $F_G(K)$ is only partial. 
\end{Cor}

\begin{Def} Assume we are in the situation of Section~\ref{Sect_functoriality_first}. Assume given isomorphisms
$$
\beta_V: \H^{\la}_H(V, \cM)\,\iso\, \H^{\ra}_G(V, \cM)
$$
in $\D^{\prec}(\Bun_{G, H}\times X)$ functorial in $V\in \Rep(\check{H})$.  We require properties $\cH$1), $\cH$2) below. First, iterating $\beta_V$, one gets isomorphisms over $\Bun_{G,H}\times X^n$
$$
\beta_{V_1,\ldots, V_n}: \H^{\la}_H(V_1\boxtimes\ldots\boxtimes V_n, \cM)\,\iso\, \H^{\ra}_G(V_1\boxtimes\ldots\boxtimes V_n, \cM)
$$
Here by $\H^{\ra}_G(V, \cM)$ we mean $\H^{\ra}_G(\Res^{\kappa}(V), \cM)$, and similarly for the iterated isomorphisms. It is required that for $V_1, V_2\in\Rep(\check{H})$ the following diagrams commute\\
$\cH$1)
$$
\begin{array}{ccc}
\H^{\la}_H(V_1\boxtimes V_2,\cM) & \toup{\beta_{V_1, V_2}} & \H^{\ra}_G(V_1\boxtimes V_2, \cM)\\
\downarrow && \downarrow\\
(\id\times s)^*\H^{\la}_H(V_2\boxtimes V_1, \cM) & \toup{(\id\times s)^*\beta_{V_2, V_1}} & (\id\times s)^*\H^{\ra}_G(V_2\boxtimes V_1, \cM)
\end{array}
$$
$\cH$2)
$$
\begin{array}{ccc}
\H^{\la}_H(V_1\boxtimes V_2,\cM)\mid_{\Bun_{G,H}\times\triangle(X)} & \toup{\beta_{V_1, V_2}} & H^{\ra}_G(V_1\boxtimes V_2, \cM)\mid_{\Bun_{G,H}\times\triangle(X)} \\
\downarrow && \downarrow\\
\H^{\la}_H(V_1\otimes V_2, \cM)[1] & \toup{\beta_{V_1\otimes V_2}} & \H^{\ra}_G(V_1\otimes V_2, \cM)[1]
\end{array}
$$ 
Here the vertical arrows are natural isomorphisms. In this case we say that $\cM$ satisfies the Hecke property for $\kappa$. 
\end{Def}

\begin{Thm} 
\label{Thm_Hecke_propety_of_cM_implies}
Assume $\cM\in\D^{\prec}(\Bun_{G,H})$ satisfies the Hecke property for $\kappa$. Then $F_H$ (resp., $F_G$) commutes with the Hecke functors along $\kappa$ (resp., backward $\kappa$).
\end{Thm}
\begin{proof}
i) The case of $F_H$ is established precisely as in (\cite{L1}, the derivation of Theorem~3 part 1) from Theorem 4 part 1)). 

\medskip\noindent
ii) To establish the desired structure for $F_G$,
consider the diagram 
\begin{equation}
\label{diag_one}
\begin{array}{ccccc}
\Bun_H\times X & \getsup{h^{\la}_H\times\supp} & \cH_H & \toup{h_H^{\ra}} & \Bun_H\\
\uparrow\lefteqn{\scriptstyle \gq\times\id} && \uparrow && \uparrow\lefteqn{\scriptstyle \gq}\\
\Bun_{G,H}\times X & \getsup{h^{\la}_H\times\supp} & \cH_H\times \Bun_G & \toup{h_H^{\ra}} & \Bun_{G,H}\\
\downarrow\lefteqn{\scriptstyle p_G\times\id}\\
\Bun_G\times X
\end{array}
\end{equation}
similar to that of (\cite{L1}, Section~8.1). From definitions for $V\in\Rep(\check{H})$, $K\in \D_{nc}(\Bun_H)$ we get 
\begin{equation}
\label{complex_1}
F_G\H^{\la}_H(V, K)\,\iso\,
(p_G\times\id)_!(\H^{\ra}_H(V, \cM)\otimes p_H^*K)[-\dim\Bun_H]
\end{equation}
We used the fact that 
$$
(\IC(\Bun_H)\tboxtimes \ast\cS)^r\,\iso\, (\IC(\Bun_H)\tboxtimes \cS)^l
$$ 
on $\cH_H$ for $\cS\in\Sph_H$, cf. Section~\ref{Sect_def_Hecke_functors}. By our assumptions, (\ref{complex_1}) identifies with
$$
(p_G\times\id)_!(\H^{\la}_G(V, \cM)\otimes p_H^*K)[-\dim\Bun_H]
$$
By the base change and the projection formula, the latter complex identifies with $\H^{\la}_G(V, F_G(K))$.
\end{proof}

\ssec{Theta-lifting}
\label{Section_Theta-lifting_2.2}

\sssec{} For $r\ge 1$ let $G_r$ be the group scheme on $X$ of automorphisms of $\cO^r_X\oplus\Omega^r$ preserving the natural symplectic form $\wedge^2(\cO^r_X\oplus\Omega^r)\to \Omega$. The stack $\Bun_{G_r}$ classifies $M\in\Bun_{2r}$ with a symplectic form $\wedge^2 M\to\Omega$. Let $\cA_r$ be the line bundle on $\Bun_{G_r}$ with fibre $\det\RG(X,M)$ at $M$. Write $\Bunt_{G_r}\to\Bun_{G_r}$ for the $\mu_2$-gerbe of square roots of $\cA_r$. The theta-sheaf $\Aut$ on $\Bunt_{G_r}$ is defined in \cite{L2}.

\sssec{Orthogonal-symplectic dual pair}  
\label{Section_2.2.2_theta}
Let $n,m\ge 1$. Let $G=G_n$ and $H=\SO_{2m}$ split. The stack $\Bun_H$ classifies $V\in\Bun_{2m}$, a nondegenerate symmetric form $\Sym^2 V\to \cO$, and a compatible trivialization $\det V\,\iso\,\cO_X$. The theta-lifting functors 
\begin{equation}
\label{functors_F_G_F_H}
F_G: \D^-(\Bun_H)_!\to \D^{\prec}(\Bun_G), \;\;\;\;\; F_H: \D^-(\Bun_G)_!\to \D^{\prec}(\Bun_H)
\end{equation}
from (\cite{L1}, Section~2.3) are defined as follows.

 Let $\tau: \Bun_G\times\Bun_H\to\Bun_{G_{2nm}}$ be the map sending $(M,V)$ to $M\otimes V$ with the induced symplectic form $\wedge^2(M\otimes V)\to\Omega$. There is a canonical lift 
$$
\tilde\tau: \Bun_G\times\Bun_H\to\Bunt_{G_{2nm}}
$$ 
exteding $\tau$ (cf. \cite{L1}, Section~2.3.1). Let $\cM=\tilde\tau^*\Aut[\dim\Bun_G\times\Bun_H-\dim\Bun_{G_{2nm}}]$. Using the kernel $\cM$, define functors (\ref{functors_F_G_F_H}) as in Section~\ref{Sect_functoriality_first}. 

\sssec{}  
\label{Sect_kappa_for_orth-sympl}
For $m>n$ define $\kappa: \check{G}\times\Gm\to\check{H}$ as in (\cite{L1}, Section~2.3.2). Namely, let $\cW=\Qlb^{2m}$ be the standard representation of $\check{H}=\SO_{2m}$. Pick an orthogonal decomposition $\cW=\cW_1\oplus \cW_2$, where $\dim \cW_1=2n+1$. Identify $\check{G}\,\iso\,\SO(\cW_1)$, this fixes the inclusion $\check{G}\hook{}\check{H}$. The connected centralizer of $\check{G}$ in $\check{H}$ is $\SO(\cW_2)$. Let $prin: \SL_2\to \SO(\cW_2)$ be the homomorphism corresponding to a principal (irreducible) representation of $\SL_2$ on $\cW_2$. Then the restriction of $\kappa$ to $\Gm$ is the composition $\Gm\hook{}\SL_2\toup{prin}\SO(\cW_2)\hook{} \check{H}$, where the first map is the maximal torus of diagonal matrices in $\SL_2$.   
  
  For $m\le n$ define $\kappa: \check{H}\times\Gm\to \check{G}$ as in \select{loc.cit.} Namely, let $\cW=\Qlb^{2n+1}$ be the standard representation of $\check{G}=\SO_{2n+1}$. Pick an orthogonal decomposition $\cW=\cW_1\oplus \cW_2$, where $\dim\cW_1=2m$ and identify $\check{H}\,\iso\, \SO(\cW_1)$. This fixes the inclusion $\check{H}\hook{}\check{G}$. The connected centralizer of $\check{H}$ in $\check{G}$ is $\SO(\cW_2)$. Let $prin: \SL_2\to \SO(\cW_2)$ be the homomorphism corresponding to a principal (irreducible) representation of $\SL_2$ on $\cW_2$. Then the restriction of $\kappa$ to $\Gm$ is the composition $\Gm\hook{}\SL_2\toup{prin}\SO(\cW_2)\hook{} \check{G}$, where the first map is the maximal torus of diagonal matrices in $\SL_2$. 

\begin{Thm}[\cite{L1}, Theorem~4] Assume we are in the situation of 
Section~\ref{Section_2.2.2_theta}. \\
For $n,m\ge 1$ the complex $\cM$ satisfies the Hecke property for $\kappa$.
\end{Thm}  
\begin{Cor} Assume we are in the situation of Section~\ref{Section_2.2.2_theta}. 
\begin{itemize}
\item If $m\le n$ then $F_G$ (resp., $F_H$) commutes with Hecke actions along $\kappa$ (resp., backward $\kappa$).
\item If $m>n$ then $F_H$ (resp., $F_G$) commutes with Hecke actions along $\kappa$ (resp., backward $\kappa$).
\end{itemize}
\end{Cor}
\begin{proof}
This follows from Theorem~\ref{Thm_Hecke_propety_of_cM_implies}. 
The straight direction is also established in (\cite{L1}, Theorem~3).
\end{proof}

\sssec{Theta-lifting for similitudes} 
\label{Section_2.2.5_similitudes}
Let $n, m\ge 1$. Let $\GG=\GSp_{2n}$, $\HH=\GSO_{2m}=(\Gm\times\SO_{2m})/(-1,-1)$ be as in (\cite{L3}, Section~2.4.4). Here $\GG,\HH$ are split. The theta-lifting functors 
\begin{equation}
\label{functors_theta-lift_similitudes}
F_{\GG}: \D^-(\Bun_{\HH})_!\to \D^{\prec}(\Bun_{\GG}), \;\;\;\;\; F_{\HH}: \D^-(\Bun_{\GG})_!\to \D^{\prec}(\Bun_{\HH})
\end{equation}
from (\cite{L3}, Definition~2.4.8) are as follows.  

  The stack $\Bun_{\GG}$ classifies $(M\in\Bun_{2n}, \cA\in\Bun_1)$ with a symplectic form $\wedge^2 M\to\cA$ on $X$. The stack $\Bun_{\HH}$ classifies $(V\in\Bun_{2m}, \cC\in\Bun_1)$, a nondegenerate symmetric form $\Sym^2 V\to\cC$, and a compatible trivialization $\gamma: \cC^{-m}\otimes\det V\,\iso\, \cO_X$. 
  
  Let $\Bun_{\GG\HH}=\Bun_{\GG}\times_{\Bun_1}\Bun_{\HH}$, where the map $\Bun_{\GG}\to\Bun_1$ sends $(M,\cA)$ to $\cA$, and $\Bun_{\HH}\to\Bun_1$ sends $(V,\cC,\gamma)$ to $\cC^{-1}\otimes\Omega$. The notation $\Bun_{\GG\HH}$ is not to be confused with $\Bun_{\GG,\HH}=\Bun_{\GG}\times\Bun_{\HH}$. 
  
  Let $\tau: \Bun_{\GG\HH}\to \Bun_{G_{2nm}}$ be the map sending a collection $(M, \cA, V, \cC,\gamma)$ to $M\otimes V$ with the induced symplectic form $\wedge^2(M\otimes V)\to\Omega$. The map $\tau$ admits a canonical lift 
$$
\tilde\tau: \Bun_{\GG\HH}\to \Bunt_{G_{2nm}}
$$ 
defined in (\cite{L4}, Section~2.4.7). Define the complex 
$$
\Aut_{\GG\HH}=\tilde\tau^*\Aut[\dim\Bun_{\GG\HH}-\dim\Bun_{G_{2nm}}]
$$
on $\Bun_{\GG\HH}$. Let $q: \Bun_{\GG\HH}\to \Bun_{\GG, \HH}$ be the natural map. Set $\cM=q_!\Aut_{\GG\HH}$.
Define the functors (\ref{functors_theta-lift_similitudes}) as in Section~\ref{Sect_functoriality_first}.   
  
\sssec{} For $m\ge 3$ consider the group $\Spin_m$ defined in (\cite{GW}, Section~6.3.3). By (\cite{GW}, Theorem~6.3.5), it is equipped with a distinguished surjection $\Spin_m\to \SO_m$ given by the standard representation, whose kernel is denoted $\{1,\iota\}\,\iso\,\mu_2$. 
For $m\ge 3$ set $\GSpin_m=\Gm\times\Spin_m/\{(-1,\iota)\}$. We convent that $\GSpin_2=\Gm\times\Gm$. Recall the Langlands dual groups are $\check{\HH}\,\iso\, \GSpin_{2m}, \check{\GG}\,\iso\,\GSpin_{2n+1}$ (cf. \cite{L3}, Section~3). 

\sssec{} 
\label{Sect_2.2.7_def_of_bar_kappa}
For $m>n$ define $\bar\kappa: \check{\GG}\times\Gm\to\check{\HH}$ as the map $(i_{\kappa},\delta_{\kappa})$ from (\cite{L3}, Sections~2.5 and 5.12.5). 
For $m\le n$ define $\bar\kappa: \check{\HH}\times\Gm\to \check{\GG}$ as $(i_{\kappa},\delta_{\kappa})$ from (\cite{L3}, Sections~2.5 and 5.12.6). 

 For example, for $m=3, n=2$ the map $i_{\kappa}: \check{\GG}\to\check{\HH}$ identifies with the quotient of the map $\Gm\times \Sp_4\to \Gm\times\SL_4$, $(x,y)\mapsto (x^{-1}, y)$ by the diagonally embedded $(-1,-1)$. We will especially need this case.

Consider the surjections $st: \check{\GG}\to \check{G}=\SO_{2n+1}$ and $st: \check{\HH}\to \check{H}=\SO_{2m}$
given by the standard representations. For $m\le n$ consider the map $\kappa: \check{H}\times\Gm\to\check{G}$ of Section~\ref{Sect_kappa_for_orth-sympl}. Then the diagram commutes
$$
\begin{array}{ccc}
\check{\HH}\times\Gm & \toup{\bar\kappa} & \check{\GG}\\
\downarrow\lefteqn{\scriptstyle st\times\id} && \downarrow\lefteqn{\scriptstyle st}\\
\check{H}\times\Gm & \toup{\kappa} &\check{G}
\end{array}
$$
Let $C(\check{\HH})\subset \check{\GG}$ be the connected centralizer of $i_{\kappa}(\check{\HH})$ in $\check{\GG}$, $\SL_2\toup{prin} C(\check{\HH})$ be the homomoprhism corresponding to the principal unipotent orbit. Then $\delta_{\kappa}: \Gm\to \check{\GG}$ is the composition $\Gm\hook{}\SL_2\toup{prin} C(\check{\HH})\hook{}\check{\GG}$. 
 
 For $m>n$ consider the map $\kappa: \check{G}\times\Gm\to\check{H}$ from Section~\ref{Sect_kappa_for_orth-sympl}. Then the diagram commutes
$$
\begin{array}{ccc}
\check{\GG}\times\Gm & \toup{\bar\kappa} & \check{\HH}\\
\downarrow\lefteqn{\scriptstyle st\times\id} && \downarrow\lefteqn{\scriptstyle st}\\
\check{G}\times\Gm & \toup{\kappa} & \check{H}
\end{array}
$$
Let $C(\check{\GG})$ be the connected centralizer of $i_{\kappa}(\check{\GG})$ in $\check{\HH}$. Let $\SL_2\toup{prin} C(\check{\GG})$ be the homomorphism corresponding to the principal unipotent orbit. Then $\delta_{\kappa}: \Gm\to \check{\HH}$ is the composition $\Gm\hook{}\SL_2\toup{prin} C(\check{\GG})\hook{} \check{\HH}$. 

If $m=n$ or $m=n+1$ then $\delta_{\kappa}$ is trivial.

\begin{Thm} 
\label{Sect_2.2.9_similitude_cM}
Assume we are in the situation of Section~\ref{Section_2.2.5_similitudes}. Then for any $n,m$ the complex $\cM\in\D^{\prec}(\Bun_{\GG,\HH})$ satisfies the Hecke property for $\kappa$.
\end{Thm}
\begin{proof}
We derive this from (\cite{L3}, Theorem~2.5.8). We give the proof only in the case $m>n$, the case $m\le n$ is similar.

As in (\cite{L3}, Section~2.5.7), for $a\in\ZZ$ write $^a\Bun_{\GG\HH}$ for the stack classifying $(M,\cA)\in\Bun_{\GG}$, $(V,\cC,\gamma)\in\Bun_{\HH}$, $x\in X$, and an isomorphism $\cA\otimes\cC\,\iso\, \Omega(ax)$. 
We have the commutative diagram
$$
\begin{array}{ccccc}
\mathop{\sqcup}\limits_{a\in\ZZ} {^a\Bun_{\GG\HH}} & \getsup{\bar h^{\la}_{\HH}} & \cH_{\HH}\times_{\Bun_{\HH}} \Bun_{\GG\HH} & \toup{h^{\ra}_{\HH}\times\id} & \Bun_{\GG\HH}\\
\downarrow\lefteqn{\scriptstyle \bar q} && \downarrow && \downarrow\lefteqn{\scriptstyle q}\\
\Bun_{\GG,\HH}\times X & \getsup{h^{\la}_{\HH}\times\id\times\supp} & \cH_{\HH}\times \Bun_{\GG} & \toup{h^{\ra}_{\HH}\times\id} & \Bun_{\GG,\HH},
\end{array}
$$
where we used the map $h^{\ra}_{\HH}: \cH_{\HH}\to\Bun_{\HH}$ to define the fibred product $\cH_{\HH}\times_{\Bun_{\HH}} \Bun_{\GG\HH}$. We denoted by $\bar q$ the natural projection.
Here a point of $\cH_{\HH}\times_{\Bun_{\HH}} \Bun_{\GG\HH}$ is given by a collection $(x, V,\cC, V',\cC')\in \cH_{\HH}$ and $(M,\cA, V',\cC')\in\Bun_{\GG\HH}$. The map $\bar h^{\la}_{\HH}$ sends this collection to $(x, M,\cA, V,\cC)$. This gives an isomorphism for $V\in\Rep(\check{\HH})$
\begin{equation}
\label{iso_first_Rep(H)-action_Th_2.2.9}
\H^{\la}_{\HH}(V, \cM)\,\iso\, \bar q_!\H^{\la}_{\HH}(V, \Aut_{\GG\HH})
\end{equation}
By (\cite{L3}, Theorem~2.5.8),
$$
\H^{\la}_{\HH}(V, \Aut_{\GG\HH})\,\iso\, \H^{\ra}_{\GG}(V, \Aut_{\GG\HH})
$$ 
over $\mathop{\sqcup}\limits_{a\in\ZZ} {^a\Bun_{\GG\HH}}$. So, (\ref{iso_first_Rep(H)-action_Th_2.2.9}) identifies with $\bar q_!\H^{\ra}_{\GG}(V, \Aut_{\GG\HH})$. The diagram commutes
$$
\begin{array}{ccccc}
\Bun_{\GG\HH} & \getsup{h^{\la}_{\GG}\times\id} & \cH_{\GG}\times_{\Bun_{\GG}} \Bun_{\GG\HH} & \toup{\bar h^{\ra}_{\GG}} & \mathop{\sqcup}\limits_{a\in\ZZ} {^a\Bun_{\GG\HH}}\\
\downarrow\lefteqn{\scriptstyle q} && \downarrow && \downarrow\lefteqn{\scriptstyle \bar q}\\
\Bun_{\GG,\HH} & \getsup{h^{\la}_{\GG}} & \cH_{\GG}\times\Bun_{\HH} & \toup{h^{\ra}_{\GG}\times\id\times\supp} & \Bun_{\GG,\HH}\times X,
\end{array}
$$
where we used the map $h^{\la}_{\GG}: \cH_{\GG}\to\Bun_{\GG}$ to define the fibred product $\cH_{\GG}\times_{\Bun_{\GG}} \Bun_{\GG\HH}$. So, a point of $\cH_{\GG}\times_{\Bun_{\GG}} \Bun_{\GG\HH}$ is given by $(x, M,\cA, M',\cA')\in\cH_{\GG}$, $(M,\cA, V,\cC)\in \Bun_{\GG\HH}$. The map $\bar h^{\ra}_{\GG}$ sends this collection to $(x, M',\cA', V,\cC)$. This gives an isomorphism 
$$
\bar q_!\H^{\ra}_{\GG}(V, \Aut_{\GG\HH})\,\iso\, \H^{\ra}_{\GG}(V, \cM)
$$
Our claim follows. 
\end{proof}

\begin{Cor} Assume we are in the situation of Section~\ref{Section_2.2.5_similitudes}.
\begin{itemize}
\item If $m\le n$ then $F_{\GG}$ (resp., $F_{\HH}$) commutes with Hecke actions along $\bar\kappa$ (resp., backward $\kappa$).
\item If $m>n$ then $F_{\HH}$ (resp., $F_{\GG}$) commutes with Hecke actions along $\bar\kappa$ (resp., backward $\kappa$).
\end{itemize}
\end{Cor}
\begin{proof} This follows from Theorem~\ref{Thm_Hecke_propety_of_cM_implies}. 
The straight direction is also established in (\cite{L3}, Theorem~2.5.5). 
\end{proof}

\begin{Rem} 
\label{Rem_kappa_for_Spin}
For future references, we record the following. Assume we are in the situation of Section~\ref{Sect_2.2.7_def_of_bar_kappa}. Write $st: \Spin_{2n+1}\to \SO_{2n+1}$ and $st: \Spin_{2m}\to \SO_{2m}$ for the standard representations of these groups. If $m\le n$ then the preimage of $\check{H}\hook{}\check{G}$ under $st: \Spin_{2n+1}\to \SO_{2n+1}=\check{G}$ identifies with $\Spin_{2n}$. This gives the embedding $\kappa: \Spin_{2m}\hook{}\Spin_{2n+1}$. 

 If $m>n$ then the preimage of $\check{G}\hook{} \check{H}$ under $st: \Spin_{2m}\to\SO_{2m}$ identifies with $\Spin_{2n+1}$. This gives the embedding $\kappa: \Spin_{2n+1}\hook{}\Spin_{2m}$.
\end{Rem}
 
\sssec{Theta-lifting for general linear groups} 
\label{Section_2.2.9}
Let $n,m\ge 1$. Set $G=\GL_n, H=\GL_m$. In this case the theta-lifting functors (\ref{functors_F_G_F_H}) are defined by $F_G=F_{m,n}, F_H=F_{n,m}$, where $F_{n,m}$ are given in (\cite{L1}, Definition~3). Recall this definition.

  Let $\cW_{n,m}$ be te stack classifying $L\in\Bun_n, U\in\Bun_m$ and a section $\cO_X\to L\otimes U$ on $X$. Let $q: \cW_{n,m}\to \Bun_n\times\Bun_m$ be the projection sending the above point to $(L, U)$. Let 
$$
\cI=\Qlb[\dim(\Bun_n\times\Bun_m)+a_{n,m}]\in \D^b(\cW_{n,m}),
$$ 
where $a_{n,m}$ is the function of a connected component of $\Bun_n\times\Bun_m$ sending $(L,U)$ to the Euler characteristics $\chi(X, L\otimes U)$. Set $\cM=q_!\cI\in \D^b(\Bun_{G,H})$. Define the functors (\ref{functors_F_G_F_H}) as in Section~\ref{Sect_functoriality_first} using the kernel $\cM$. 

\sssec{} Assume $m\ge n$. Define $\kappa=(i_{\kappa}, \delta_{\kappa}): \check{G}\times\Gm\to\check{H}$ as follows. 
Write $\cW$ for the standard representation of $\check{H}=\GL_m$. Pick a decomposition of vector spaces $\cW=\cW_1\oplus\cW_2$ with $\dim\cW_1=n$. So, $\GL(\cW_1)\times\GL(\cW_2)\subset \check{H}$ is a Levi subgroup. Define $i_{\kappa}: \GL_n\to \GL_m$ as the composition $\GL(\cW_1)\toup{\sigma}\GL(\cW_1)\hook{}\GL(\cW)$, where $\sigma(g)=(^tg)^{-1}$. We have denoted by $^tg$ the transpose of a matrix $g$.

Let $prin: \SL_2\to \GL(\cW_2)$ be a principal (irreducible) representation of $\SL_2$. Set $\kappa=(i_{\kappa}, \delta_{\kappa})$, where $\delta_{\kappa}$ the composition
$$
\Gm\hook{}\SL_2\toup{prin} \GL(\cW_2)\hook{}\GL(\cW)
$$
Here the first map is the standard maximal torus of diagonal matrices. 
\begin{Thm} 
\label{Thm_2.2.14_cM_Hecke_GL_pair}
Assume we are in the situation of Section~\ref{Section_2.2.9}. Then $\cM$ satisfies the Hecke property for $\kappa$.
\end{Thm}
\begin{proof}
We derive this from (\cite{L1}, Theorem~6). As in (\cite{L1}, Section~2.4), let $_{\infty}\cW_{n,m}$ be the stack classifying $x\in X$, $L\in\Bun_n, U\in\Bun_m$ and a section $t: \cO_X\to L\otimes U(\infty x)$, which is allowed to have any pole at $x$. This is an ind-algebraic stack. Let $\bar q: {_{\infty}\cW_{n,m}}\to \Bun_{G,H}\times X$ be the map forgetting $t$. The Hecke functors 
$$
\H^{\la}_G, \H^{\ra}_G:
\Rep(\check{\GG}\times\Gm)\times \D^{\prec}(\cW_{n,m})\to \D^{\prec}(_{\infty}\cW_{n,m})
$$
and
$$
\H^{\la}_H, \H^{\ra}_H:
\Rep(\check{\HH}\times\Gm)\times \D^{\prec}(\cW_{n,m})\to \D^{\prec}(_{\infty}\cW_{n,m})
$$
are defined in (\cite{L1}, Section~7). 
By (\cite{L1}, Theorem~6), for $V\in\Rep(\check{H})$ one has
$$
\H^{\la}_H(V, \cI)\,\iso\, \H^{\ra}_G(V, \cI)
$$
in $\D^{\prec}(_{\infty}\cW_{n,m})$. Now the argument analogous to our proof of Theorem~\ref{Sect_2.2.9_similitude_cM} gives the desired claim.
\end{proof}

\begin{Cor}
Assume we are in the situation of Section~\ref{Section_2.2.9} with $m\ge n$. Then $F_H$ (resp., $F_G$) commutes with Hecke actions along $\kappa$ (resp., backward $\kappa$).
\end{Cor}
\begin{proof}
This follows from Theorem~\ref{Thm_Hecke_propety_of_cM_implies}. The straight direction is also established in (\cite{L1}, Theorem~5).
\end{proof}

\ssec{Geometric Eisenstein series}

\sssec{} Let $H$ be a connected reductive group over $k$ with a given maximal torus and Borel subgroup. Let $G$ be a standard Levi subgroup of $H$, so we have the corresponding inclusion of dual groups $\check{G}\hook{}\check{H}$. Extend it to a map $\kappa: \check{G}\times\Gm\to\check{H}$ trivial on $\Gm$. Let $P\subset G$ be a standard parabolic subgroup with Levi quotient $M$. Consider the Drinfeld compactification $\Bunt_P$ defined in (\cite{BG}, Section~1.3.6). We have the diagram of projections $\Bun_M\getsup{\gq}\Bunt_P\toup{\gp}\Bun_G$. Let $q: \Bunt_P\to\Bun_M\times\Bun_G$ be the map $\gq\times\gp$, set $\cM=q_!\IC_{\Bunt_P}$. Define the functor $F_H$ using $\cM$ as in Section~\ref{Sect_functoriality_first}. 
\begin{Pp} The complex $\cM$ satisfies the Hecke property for $\kappa$.
\end{Pp}
\begin{proof}
This is derived from (\cite{BG}, Corollary~4.1.7) as in our proof of Theorems~\ref{Sect_2.2.9_similitude_cM} and \ref{Thm_2.2.14_cM_Hecke_GL_pair}.
\end{proof}

\ssec{Liftings of Hecke eigen-sheaves}
\label{Sect_Liftings of Hecke eigen-sheaves}

In all the above cases of theta-lifting and geometric Eisenstein series we can derive from Corollary~\ref{Cor_2.1.8_partial_Hecke} the results like the following.
\begin{Cor} 
\label{Cor_2.4.1_Lifting_orth_sympl}
Assume we are in the situation of Section~\ref{Section_2.2.2_theta}. So, $n,m\ge 1$, $G=G_n$, $H=\SO_{2m}$. 

\smallskip\noindent
i) Let $m>n$, $x\in X$. Let $\sigma: \pi_1(X,x)\times\Gm\to \check{H}$ be a homomorphism. Let $K\in\D^-(\Bun_H)_!$ be a $\sigma$-Hecke eigen-sheaf. Then $F_G(K)$ satisfies the $\sigma$-Hecke propety with respect to $\kappa: \check{G}\times\Gm\to\check{H}$. 

\smallskip\noindent
ii) Let $m\le n$, $x\in X$. Let $\sigma: \pi_1(X,x)\times\Gm\to \check{G}$ be a homomorphism. Let $K\in\D^-(\Bun_G)_!$ be a $\sigma$-Hecke eigen-sheaf. Then $F_H(K)$ satisfies the $\sigma$-Hecke propety with respect to $\kappa: \check{H}\times\Gm\to \check{G}$. 
\end{Cor}
\begin{Cor} 
\label{Cor_2.4.2_Lifting_similitudes}
Assume we are in the situation of Section~\ref{Section_2.2.5_similitudes}. So, $n, m\ge 1$ and $\GG=\GSp_{2n}$, $\HH=\GSO_{2m}$. 

\smallskip\noindent
i) Let $m>n$, $x\in X$. Let $\sigma: \pi_1(X,x)\times\Gm\to \check{H}$ be a homomorphism. Let $K\in\D^-(\Bun_H)_!$ be a $\sigma$-Hecke eigen-sheaf. Then $F_G(K)$ satisfies the $\sigma$-Hecke propety with respect to $\kappa: \check{G}\times\Gm\to\check{H}$. 

\smallskip\noindent
ii) Let $m\le n$, $x\in X$. Let $\sigma: \pi_1(X,x)\times\Gm\to \check{G}$ be a homomorphism. Let $K\in\D^-(\Bun_G)_!$ be a $\sigma$-Hecke eigen-sheaf. Then $F_H(K)$ satisfies the $\sigma$-Hecke propety with respect to $\kappa: \check{H}\times\Gm\to \check{G}$. 
\end{Cor}
  
\section{More results} 
\label{Sect_Extending the Hecke property}

In Secton~\ref{Sect_Extending the Hecke property} we formulate 
our main results that have not appeared in Section~\ref{Sect_functoriality}. They are
related to the extension of a partial Hecke property and the construction of new automorphic sheaves. Their proofs are collected in Section~\ref{Sect_Proofs}. 

\ssec{Extending the Hecke property}Ê
\label{Sect_Extending_Hecke_3.1}

\sssec{} In Appendix~\ref{Appendix_B} we study the following question. Given an inclusion of connected reductive groups $\kappa: \check{G}\hook{}\check{H}$ over $\Qlb$, an abelian category $\cC$ with a $\Rep(\check{G})$-action, and a $\Rep(\check{H})$-Hecke eigen-object $c$ of $\cC$, can one extend this partial Hecke property to a $\Rep(\check{G})$-Hecke property of $c$? The answer in some sense is given by Proposition~\ref{Pp_A.2.3}. The results of this section are inspired by Proposition~\ref{Pp_A.2.3}.

\begin{Pp} 
\label{Pp_2.2.1} Let $G,H$ be connected reductive groups over $k$, $\kappa: \check{G}\hook{}\check{H}$ be a closed subgroup, $E_{\check{G}}$ be a $\check{G}$-local system on $X$, $E$ be the $\check{H}$-local system on $X$ induced via $\kappa$. Assume for any irreducible representation $V^{\lambda}$ of $\check{G}$ there is $W\in\Rep(\check{H})$ such that $V^{\lambda}$ appears in $\Res^{\kappa}(W)$ with multiplicity one. Let $K\in\D^{\prec}(\Bun_G)$ be equipped with a $E$-Hecke property with respect to $\kappa$. 
(It is understood that $\kappa$ is extended to a morphism $\check{G}\times\Gm\to \check{H}$ trivial on $\Gm$). Then there could exist at most one extension of this structure to a structure of a $E_{\check{G}}$-Hecke eigen-sheaf on $K$. 
\end{Pp}

\begin{Rem}
\label{Rem_actual_embeddings_one}
The assumptions of Proposition~\ref{Pp_2.2.1} are satisfied for the following embeddings $\kappa$:
\begin{itemize}
\item[A1)] $\GL_{n-1}\hook{}\GL_n$ given as the subgroup of matrices of the form 
$\left(
\begin{array}{cc}
y & 0\\
0 & 1
\end{array}
\right)
$. 
\item[A2)] for $n\ge 2$ the inclusion $\Spin_{2n-1}\hook{}\Spin_{2n}$ of Remark~\ref{Rem_kappa_for_Spin}.
\item[A3)] for $n\ge 2$ the inclusion $\Spin_{2n}\hook{}\Spin_{2n+1}$ 
of Remark~\ref{Rem_kappa_for_Spin}.
\item[A4)] for $n\ge 2$ the inclusion $\GSpin_{2n-1}\hook{}\GSpin_{2n}$ given in Section~\ref{Sect_2.2.7_def_of_bar_kappa}.
\item[A5)] for $n\ge 1$ the natural inclusion $\Sp_{2n}\hook{}\SL_{2n}$. 
\item[A6)] let $\check{G}, \check{G}_1$ be connected reductive groups over $\Qlb$ with a given homomorphism $\check{G}\to \check{G}_1$, write $\kappa: \check{G}\hook{}\check{G}\times\check{G}_1$ for its graph. 
\end{itemize}
\end{Rem}

 
  We give an example where the $E$-Hecke property of $K$ with respect to $\kappa$ in Proposition~\ref{Pp_2.2.1} does not extend to a structure of a $E_{\check{G}}$-Hecke property (cf. Remark~\ref{Rem_counterexample}). One can always twist the $E$-Hecke property of $K$ by an element of the center of $\check{H}$ as in  Remark~\ref{Rem_counterexample}. If the original $E$-Hecke property of $K$ extends to a $E_{\check{G}}$-Hecke property, this is not necessarily the case after this twisting. 
  
\begin{Pp} 
\label{Pp_2.2.3}
In the situation of Proposition~\ref{Pp_2.2.1} assume one of the following:
\begin{itemize}
\item[1)] $\kappa: \check{G}\hook{}\check{H}$ is the inclusion A1) with $n\ge 2$. View $E_{\check{G}}$ as a rank $n-1$ local system $E_0$ on $X$. Assume 
$$
\H^0(X, E_0)=\H^0(X, E_0^*)=0
$$
\item[2)] $\kappa: \check{G}\hook{}\check{H}$ is the inclusion A6).
\item[3)] $\kappa: \check{G}\hook{}\check{H}$ is the inclusion $\Sp_{2n}\hook{} \SL_{2n}$. View $E_{\check{G}}$ as a rank $2n$ local system $E$ on $X$ with a symplectic form $\omega: \wedge^2 E\to \Qlb$. Let $\cW=\Ker\omega$. Assume 
$$
\H^0(X,\cW)=\H^0(X, \cW^*)=0
$$
\item[4)] Let $G=\GSp_4$, $H=\GSO_6$ split, and $\kappa: \check{G}\hook{}\check{H}$ be the inclusion A4) for $n=3$. View $E_{\check{G}}$ as a pair $(E,\chi)$, where $E$ (resp., $\chi$) is a rank 4 (resp., rank one) local system on $X$ with a symplectic form $\omega: \wedge^2 E\to \chi$. Set $\cW=\Ker\omega$. Assume 
$$
\H^0(X, \cW\otimes\chi^{-1})=\H^0(X, \cW^*\otimes\chi)=0
$$
\end{itemize}
In cases 1), 2) the complex $K\in\D^{\prec}(\Bun_G)$ admits a natural $E_{\check{G}}$-Hecke property (which does not necessarily extend the $E$-Hecke property of $K\in\D^{\prec}(\Bun_G)$ with respect to $\kappa$). 

In case 3) there is a direct sum decomposition 
$$
K\,\iso\,\mathop{\oplus}\limits_{a\in\mu_n} \, K_a
$$ 
in $\D^{\prec}(\Bun_G)$ compatible with the $E$-Hecke property of $K$ with respect to $\kappa$. For each $a\in\mu_n$, $K_a$ can be naturally equipped with a $E_{\check{G}}$-Hecke property.

In case 4) there is a direct sum decomposition 
$$
K\,\iso\,\mathop{\oplus}\limits_{a\in\mu_2}\,  K_a
$$ 
in $\D^{\prec}(\Bun_G)$ compatible with the $E$-Hecke property of $K$ with respect to $\kappa$. For each $a\in\mu_2$, $K_a$ can be naturally equipped with a $E_{\check{G}}$-Hecke property.
\end{Pp}

\begin{Rem} i) Assume that we are in the situation of Proposition~\ref{Pp_2.2.3} 1). Then the set of liftings of  $E_{\check{H}}$ to a $\check{G}$-local system identifies with $\Qlb^*$. 

\smallskip\noindent
ii) Assume that we are in the situation of Proposition~\ref{Pp_2.2.3} 3). So, $E$ is a $\SL_{2n}$-local system on $X$ equipped with a distinguished lifting to a $\check{G}$-local system. This lifting is given by the symplectic form $\omega: \wedge^2 E\to\Qlb$. Then any other lifting of $E$ to a $\check{G}$-local system on $X$ is given by the symplectic form $a\omega: \wedge^2 E\to\Qlb$ with $a\in\mu_n$. Besides, for any $a\in\mu_n$ the $\check{G}$-local systems $(E,\omega)$ and $(E, a\omega)$ are isomorphic. Namely, a choice of $b\in \mu_{2n}\subset\Qlb$ with $b^2=a$ provides such an isomorphism $(E,\omega)\to (E, a\omega)$ given by the multiplication by $b$. 

 Assume for a moment that we work in the setting of $\cD$-modules, so that we have the corresponding algebraic stacks $\LocSys_{\check{G}}, \LocSys_{\check{H}}$ of local systems on $X$ and the induction map $\LocSys_{\check{G}}\to \LocSys_{\check{H}}$. Then the fibre of the latter map over $E_{\check{H}}$ identifies with the scheme $\mu_n$.
\end{Rem}

\ssec{Applications}
 
\sssec{} Let $m\ge 2$, $n=m-1$. Let $(G,H)$ and $\kappa: \check{G}\hook{}\check{H}$ be as in Section~\ref{Section_2.2.9} (the factor $\Gm$, on which $\kappa$ is trivial, is omitted). So, $\kappa: \GL_{m-1}\hook{}\GL_m$ is the inclusion A1). 

\begin{Cor} 
\label{Cor_2.3.2}
Let $x\in X$. Let $E_0$ be a rank $m-1$ local system on $X$, $E=E_0\oplus\Qlb$ be the induced $\check{H}$-local system on $X$. Assume $\H^0(X, E_0)=\H^0(X, E_0^*)=0$. Let $K\in\D^-(\Bun_H)_!$ be equipped with a structure of a $E$-Hecke eigen-sheaf. Then $F_G(K)$ is naturally equipped with a $E_0$-Hecke property.
\end{Cor} 
\begin{proof} Combine Proposition~\ref{Pp_2.2.3} 1),  Theorem~\ref{Thm_2.2.14_cM_Hecke_GL_pair}, and Theorem~\ref{Thm_Hecke_propety_of_cM_implies}. 
\end{proof}

\begin{Cor} Assume we are in the situation of Section~\ref{Section_2.2.5_similitudes} with $m=2, n=1$. So, $\GG=\GL_2$, $\HH=\GSO_4$, and $\kappa: \check{\GG}\hook{}\check{\HH}$ is trivial on the factor $\Gm$, which is omitted. Let $E_{\check{\GG}}$ be a $\check{\GG}$-local system on $X$, $E$ be the $\check{\HH}$-local system on $X$ induced via $\kappa$. Let $K\in\D^-(\Bun_{\HH})_!$ be equipped with a structure of a $E$-Hecke eigen-sheaf. Then $F_{\GG}(K)$ is naturally equipped with a $E_{\check{\GG}}$-Hecke property. 
\end{Cor}
\begin{proof} The argument is similar to Proposition~\ref{Pp_2.2.3} 2) using Corollary~\ref{Cor_2.4.2_Lifting_similitudes}. There is a homomorphism $\nu: \check{\HH}\to \check{\GG}$ such that $\nu\kappa=\id$. 
\end{proof}

\ssec{\bf Automorphic sheaves for $\GSp_4$} 
\label{Sect_case_GSp4}

\sssec{} 
\label{Sect_3.3.1_results_for_GSp_4}
Use notations of Section~\ref{Section_2.2.5_similitudes} with $m=3, n=2$. So, $\HH=\GSO_6$, $\GG=\GSp_4$. Let $E_{\check{\GG}}$ be a $\check{\GG}$-local system on $X$ viewed as a pair $(E,\chi)$, where $E$ (resp., $\chi$) is a rank 4 (resp., rank 1) local system on $X$ with a symplectic form $\omega: \wedge^2 E\to\chi$. Let $\kappa: \check{\GG}\hook{}\check{\HH}$ be the map denoted by $i_{\kappa}$ in Section~\ref{Sect_2.2.7_def_of_bar_kappa}. 
Let $E_{\check{\HH}}$ be the $\check{\HH}$-local system on $X$ obtained from $E_{\check{\GG}}$ by the extension of scalars via $\kappa: \check{\GG}\hook{}\check{\HH}$. Set $\cW=\Ker\omega$.

 For the convenience of the reader recall that $\check{\HH}\,\iso\,\{(c,b)\in\Gm\times\GL_4\mid \det b=c^2\}$. We may view $E_{\check{\HH}}$ as the pair $(E,\chi)$, where we forget the symplectic form but keep the induced isomorphism $\det E\,\iso\, \chi^2$ of local systems on $X$.  
\begin{Cor} 
\label{Cor_2.4.1}
Assume in the situation of Section~\ref{Sect_3.3.1_results_for_GSp_4} that $\H^0(X, \cW\otimes\chi^{-1})=\H^0(X, \cW^*\otimes\chi)=0$. Let $K\in\D^-(\Bun_{\HH})_!$ be equipped with a structure of a $E$-Hecke eigen-sheaf. Then there is a decomposition $F_{\GG}(K)\,\iso\, \oplus_{a\in\mu_2} \cK_a$ in $\D^{\prec}(\Bun_{\GG})$ such that for each $a\in\mu_2$, $\cK_a$ is naturally equipped with a $E_{\check{\GG}}$-Hecke property. \QED
\end{Cor}
\begin{proof}
Combine Proposition~\ref{Pp_2.2.3} 4) and Corollary~\ref{Cor_2.4.2_Lifting_similitudes} i). 
\end{proof} 
 
\sssec{} 
\label{Sect_2.4.2}
In the situation of Section~\ref{Sect_3.3.1_results_for_GSp_4} suppose  in addition that $E$ is an irreducible rank 4 local system on $X$. Under this assumption we have constructed a perverse sheaf denoted $K_{E,\chi, \HH}$ on $\Bun_{\HH}$ in (\cite{L4}, Lemma~17). This is a $E_{\check{\HH}}$-Hecke eigen-sheaf. 

 The following is our main result, it essentially establishes (\cite{L4}, Conjecture~6(ii)).
  
\begin{Thm} 
\label{Thm_2.3.8}
i) Under the assumptions of Section~\ref{Sect_2.4.2}, there is a decomposition 
\begin{equation}
\label{complex_for_Th_2.3.8}
F_G(K_{E,\chi, \HH})\,\iso\, \mathop{\oplus}\limits_{a\in\mu_2}\cK_a
\end{equation}
in $\D^{\prec}(\Bun_{\GG})$. For $a\in\mu_2$, $\cK_a$ is naturally equipped with a $E_{\check{\GG}}$-Hecke property. \\
ii) The complex (\ref{complex_for_Th_2.3.8}) is nonzero. So, there is a nonzero $E_{\check{\GG}}$-Hecke eigen-sheaf in $\D^{\prec}(\Bun_{\GG})$.  
\end{Thm}
\begin{proof} i) By construction, $K_{E,\chi,\HH}$ is a $E_{\check{\HH}}$-Hecke eigen-sheaf on $\Bun_{\HH}$ in the sense of Definition~\ref{Def_1.1.2}. An irreducible local system on $X$ may admit at most a unique (up to a scalar) nondegenerate bilinear form with values in a rank one local system on $X$. This implies $\H^0(X, \cW\otimes\chi^{-1})=\H^0(X, \cW^*\otimes\chi)=0$. Part i) follows now from Corollary~\ref{Cor_2.4.1}.

\smallskip\noindent
ii) In Section~\ref{sect_finite_field} we check that (\ref{complex_for_Th_2.3.8}) is nonzero provided that $X$ comes from a curve $X_0$ defined over a finite subfield $k_0\subset k$, and $E_{\check{\GG}}$ comes from a $\check{\GG}$-local system $E_{0,\check{\GG}}$ over $X_0$. Actually (\ref{complex_for_Th_2.3.8}) is always nonzero, as is shown in (\cite{L5}, Theorem~2.6.2). 
\end{proof}
 
\begin{Rem} 
\label{Rem_2.3.9}
i) For $G=\GSp_4$ the geometric Bessel periods of an object of $\D(\Bun_G)$ with a given central character are introduced in (\cite{L4}, Definition~11). A conjectural description of these Bessel periods for any Hecke eigen-sheaf in $\D(\Bun_G)$ was proposed in (\cite{L4}, Conjecture~4). The geometric Bessel periods of $F_G(K_{E,\chi, \HH})$ from Theorem~\ref{Thm_2.3.8} are described in terms of the generalized Waldspurger periods of $K_{E,\chi, \HH}$ in (\cite{L4}, Proposition 11) and further studied in \cite{L5}.
 
\smallskip\noindent
ii) For the construction of $K_{E,\chi, \HH}$ on $\Bun_{\HH}$ in (\cite{L4}, Lemma~17 and Definition~8) we used the perverse sheaf $\Aut_E$ on $\Bun_4$ normalized as in \cite{FGV02}. However, $\Aut_E$ is a $E^*$-Hecke eigen-sheaf in the sense of our Definition~\ref{Def_1.1.2}. This agrees with the fact that the geometric theta-lifting functors for the pair of split groups $(\GSp_{2n}, \GSO_{2m})$ send a complex with a given central character to a complex with the opposite central character, cf (\cite{L4}, Remark~2). 
\end{Rem}
 
  The first Whittaker coefficients functor $\Whit: \D^{\prec}(\Bun_{\GG})\to\D^-(\Spec k)$ for $\GG$ is defined in (\cite{LL}, Definition~1). 
  
\begin{Con} In the situation of Theorem~\ref{Thm_2.3.8} the complex $F_G(K_{E,\chi, \HH})$ on $\Bun_{\GG}$ is of the form $\cK\otimes \cE$, where $\cE$ is a contant complex, and $\cK$ is a perverse sheaf irreducible of each connected component of $\Bun_{\GG}$. Moreover, the first Whittaker coefficitient $\Whit(\cK)$ identifies with $\Qlb$ (up to a cohomological shift).
\end{Con}

\section{Proofs}
\label{Sect_Proofs}


\sssec{} Given complexes $K_1, K_2, K'_1, K'_2\in \D^{\prec}(S)$ for some algebraic stack $S$ and a map $f: K_1\oplus K_2\to K'_1\oplus K'_2$ in $\D^{\prec}(S)$, say that it is diagonal with respect to this decomposition if it is a sum of maps $K_i\to K'_i$ in $\D^{\prec}(S)$. This is a property of $f$, not an additional structure. 
\begin{Lm} 
\label{Lm_4.0.2}
Let $S$ be an algebraic stack locally of finite type, $K_i, K'_i\in \D^{\prec}(S)$. Let $g: K_1\oplus K_2\,\iso\, K'_1\oplus K'_2$ be an isomorphism. Write $g_{12}: K_1\to K'_2$ for the corresponding component of $g$. Assume $\H^m(g_{12}): \H^m(K_1)\to \H^m(K'_2)$ vanishes for all $m$. Then the components $g_{ii}: K_i\to K'_i$ of $g$ for $i=1,2$ are isomorphisms in $\D^{\prec}(S)$. 
\end{Lm}
\begin{proof}
For $m\in\ZZ$ the isomorphism $\H^m(g): \H^m(K_1)\oplus H^m(K_2)\,\iso\, \H^m(K'_1)\oplus H^m(K'_2)$ is triangular, so $\H^m(g_{ii})$ are isomorphisms for $i=1,2$. Our claim follows from the property (P) given in Section~\ref{sect_Notation}. 
\end{proof} 

\begin{Lm} 
\label{Lm_idempotents}
Let $S$ be an algebraic stack locally of finite type, $K\in \D^{\prec}(S)$, $q: S\times X\to S$ the projection. Let $f: q^*K\to q^*K$ be a map in $\D^{\prec}(S\times X)$, which is an idempotent. Then there is an idempotent $\bar f: K\to K$ in $\D^{\prec}(S)$ such that $q^*\bar f=f$.
\end{Lm}
\begin{proof}
1) The category $\D^{\prec}(S)$ is idempotent complete. Indeed, the $\DG$-category underlying $\D_{nc}(S)$ is cocomplete, so a retract of an object $F\in\D^{\prec}(S)$ is given by a direct summand of $F$ in $\D_{nc}(S)$. The latter lies in $\D^{\prec}(S)$ by its definition. 

For a morphism $g: q^*K\to q^*K$ in $\D^{\prec}(S\times X)$
let $\tilde g: K\to q_*q^*K=K\otimes\RG(X,\Qlb)$ be the morphism in $\D^{\prec}(S)$ corresponding to $g$ by adjointness. Given $f_i: q^*K\to q^*K$, 
the map $\wt{f_1f_2}: K\to K\otimes\RG(X,\Qlb)$ is the composition 
$$
K\toup{\tilde f_2} K\otimes\RG(X,\Qlb)\toup{\tilde f_1} K\otimes \RG(X,\Qlb)\otimes\RG(X,\Qlb)\toup{\id\times m} K\otimes \RG(X,\Qlb),
$$
where $m$ is the product in $\RG(X,\Qlb)$ given as the composition $\RG(X,\Qlb)\otimes\RG(X,\Qlb)\to \RG(X^2,\Qlb)\toup{\triangle^*} \RG(X,\Qlb)$. Here $\triangle: X\to X^2$ is the diagonal. To see this, we used the following. For the map $\bar q: X\to \Spec k$ let $\epsilon: \bar q^*\RG(X,\Qlb)\to\Qlb$ be the map on $X$ corresponding by adjointness to $\id: \RG(X,\Qlb)\to\RG(X,\Qlb)$. 
Then $\bar q_*\epsilon: \RG(X,\Qlb)\otimes \RG(X,\Qlb)\to\RG(X,\Qlb)$ equals $m$.

 So, our $f$ corresponds to $\tilde f: K\to K\otimes\RG(X,\Qlb)$ by adjoitness. Write $\tilde f=\sum_{i=0}^2 h_i$, where $h_i: K\to K\otimes \H^i(X,\Qlb)[-i]$ is the corresponding component. It is clear that $h_0: K\to K$ is an idempotent. Let $K=K_0\oplus K_1$ be the decomposition of $K$ such that $h_0$ acts by 0 (resp., by 1) on $K_0$ (resp., on $K_1$). Similarly, we have the decomposition $q^*K=F_0\oplus F_1$ of $q^*K$ under $f$. Let $p_i: q^*K\to F_i$ be the corresponding projection. We claim that the composition $q^*K_i\hook{} q^*K\toup{p_i} F_i$ is an isomorphism in $\D^{\prec}(S\times X)$. Indeed, this is an isomorphism after passing to any cohomology sheaf on $X$. Our claim follows now from (P) in Section~\ref{sect_Notation}. 
\end{proof}

\ssec{Proof of Proposition~\ref{Pp_2.2.1}} For an irreducible $V^{\lambda}\in\Rep(\check{G})$ pick $W\in\Rep(\check{H})$ such that $\Res^{\kappa}(W)\,\iso\, V^{\lambda}\oplus V'$ in $\Rep(\check{G})$, and $V^{\lambda}$ does not appear in $V'$. There could exist at most one isomorphism $\alpha_{V}$ given by (\ref{iso_for_Hecke_key}) for $V=V^{\lambda}$ such that the corresponding isomorphism $\alpha_W$ is diagonal with respect to the above decomposition of $\Res(W)$. This defines uniquely the desired isomorphism $\alpha_V$ for any $V\in\Rep(\check{G})$.
The commutations of diagrams H1), H2) is a property, not an additional structure. So, the $E$-Hecke property admits at most a unique extension to a $E_{\check{G}}$-Hecke property of $K$. 

 The necessary conditions for this extension to exist are: i) for $V^{\lambda}$, $W$ as above the isomorphism $\alpha_W$ is diagonal with respect to the decomposition $\Res^{\kappa}(W)\,\iso\, V^{\lambda}\oplus V'$; ii) the diagrams H1), H2) commute.
\QED

\sssec{Proof of Remark~\ref{Rem_actual_embeddings_one}} For the embeddings A1-A3) this follows from the branching rules (\cite{GW}, Section 8.1.1). For A4) this follows from A2). For A5) this follows from \cite{Su} (cf. also \cite{ST}). The case A6) is easy, as the composition $\check{G}\toup{\kappa}\check{G}\times\check{G}_1\toup{\pr}\check{G}$ is the identity. \QED

\begin{Rem} 
\label{Rem_counterexample} In the situation of Proposition~\ref{Pp_2.2.1}, the $E$-Hecke property of $K\in\D^{\prec}(\Bun_G)$ with respect to $\kappa$ does not always extend to a $E_{\check{G}}$-Hecke property. Consider $\kappa: \GL_1\hook{}\GL_2$ given by A1). Assume $E_{\check{G}}$ trivial, $K=\Qlb$ on $\Bun_{\Gm}$. Pick $h\in \check{H}(\Qlb)$. Equip $K$ with the isomorphisms $\alpha_V$ for each $V\in\Rep(\check{H})$ obtained as the compositions 
$$
\H^{\la}_G(V, K)\,\iso\, V[1]\boxtimes K\,\toup{h\boxtimes\id}V[1]\boxtimes K,
$$
where the first map is the tautological $E$-Hecke property of $K$, and the second one comes from the action of $h$ on $V$. This $E$-Hecke property does not extend to a $E_{\check{G}}$-Hecke property unless $h$ is diagonal.
\end{Rem}

\ssec{Reformulated Hecke property for $\GL_n$} 
\label{Sect_reformulation_GL_n}
Let $G=H=\GL_n$. Let $E$ be a $\GL_n$-local system on $X$. In \cite{G3, FGV02} it was shown that the following reformulation of the $E$-Hecke property of $K\in\D(\Bun_{\GL_n})$ is equivalent to Definition~\ref{Def_1.1.2}:

\begin{Def}
\label{Def_special_case_GL_n}
An object $K\in\D^{\prec}(\Bun_G)$ is equipped with a $E$-Hecke property if we are given an isomorphism $\alpha_V$ as in Definition~\ref{Def_1.1.2} only for the standard representation $V_0$ of $\check{H}$, for which H1) of Definition~\ref{Def_1.1.2} holds with $V_1=V_2=V_0$.
\end{Def}

\ssec{Weyl's construction for symplectic groups} 
\label{Section_5.3_Weyl_construction}
For the convenience of the reader, we recall some facts about Weyl's construction for $\Sp_{2n}$ given in (\cite{FH}, Section~17.3) and (\cite{GW}, Chapter~10). It is used in our proof of Proposition~\ref{Pp_2.2.3} below. 

\sssec{} Pick $n\ge 1$. Let $\kappa: \check{G}\hook{}\check{H}$ be the inclusion $\Sp_{2n}\hook{}\SL_{2n}$. Let $V$ be the standard representation of $\check{H}$. Pick a maximal torus and a system of positive roots in $\Sp_{2n}$ as in (\cite{FH}, Section~16.1). Write $\Lambda_G$ for the lattice of weights of $\check{G}$, $\Lambda^+_G$ for dominant weights of $\check{G}$.  So, $\Lambda_G=\ZZ^n$ and 
$$
\Lambda^+_G=\{(a_1,\ldots, a_n)\in\Lambda_G\mid a_1\ge\ldots\ge a_n\ge 0\}
$$
For $\lambda\in\Lambda^+_G$ write $V^{\lambda}$ for the irreducible representation of $\check{G}$ with highest weight $\lambda$. 

 For $d\ge 0$ let $Par(d,n)=\{\lambda\in\Lambda^+_G\mid \sum_i a_i=d\}$. The notation $Par$ stands for \select{partition}. For $\lambda\in Par(d,n)$ let $W^{\lambda}$ denote an irreducible $S_d$-module associated to $\lambda$ by the Schur-Weyl duality normalized as in (\cite{FH}, Theorem~4.3). For example, if $\lambda=(d,0,\ldots,0)$ then $W^{\lambda}$ is trivial. If $\lambda=(1,\ldots,1)$ then $W^{\lambda}$ is the sign representation. It is understood that $S_0$ is the trivial group. 
 
  For $d\ge 0$ the subspace $\cH(V^{\otimes d})\subset V^{\otimes d}$ of harmonic tensors is defined as the intersection of kernels of the operators $C_{ij}$ for $1\le i<j\le d$. Here $C_{ij}: V^{\otimes d}\to V^{\otimes d-2}$ is given by
$$
C_{ij}(v_1\otimes\ldots\otimes v_d)=\omega(v_i, v_j)v_1\otimes\ldots \hat v_i\otimes\ldots\otimes \hat v_j\otimes\ldots\otimes v_d
$$
for $v_i\in V$. Here $\omega: \wedge^2 V\to \Qlb$ is the symplectic form. By (\cite{GW}, Theorem~10.2.7), one has a canonical decomposition as a $\check{G}\times S_d$-module
\begin{equation}
\label{decomp_of_harmonics}
\cH(V^{\otimes d})\,\iso\, \mathop{\oplus}\limits_{\lambda\in Par(d,n)} V^{\lambda}\otimes W^{\lambda}
\end{equation}
 
 Let $\cB_d=\{x\in\End(V^{\otimes d})\mid xg=gx\;\,\mbox{for}\; g\in \check{G}\}$ be the centralizer algebra of the $\check{G}$-action on $V^{\otimes d}$. By (\cite{GW}, Theorem~4.2.1 and Section 10.1.1)
one has a canonical decomposition
\begin{equation}
\label{finer_decomposition}
V^{\otimes d}\,\iso\, \mathop{\oplus}\limits_{0\le r\le \frac{d}{2}} \;( 
\mathop{\oplus}\limits_{\lambda\in Par(d-2r, n)} V^{\lambda}\otimes F^{\lambda}_d),
\end{equation}
where each $F^{\lambda}_d$ is a nonzero irreducible representation of $\cB_d$. If $F^{\lambda}_d\,\iso\, F^{\lambda'}_d$ as  $\cB_d$-modules in this decomposition then $\lambda=\lambda'$, that is, (\ref{finer_decomposition}) is the isotypic decomposition under the action of $\cB_d$. The summand in (\ref{finer_decomposition}) corresponding to $r=0$  is $\cH(V^{\otimes d})$. So, if $\lambda\in Par(d,n)$ then $F^{\lambda}_d\,\iso\,W^{\lambda}$ as $S_d$-modules.

 Let $\rho_d: S_d\to\GL(V^{\otimes d})$ be the natural representation. Write $P^0$ for the composition $V^{\otimes 2}\toup{\omega}\Qlb\hook{}V^{\otimes 2}$, the second map being the canonical $\check{G}$-invariant inclusion. For $d\ge 2$ let $P$ denote the map $\id\otimes P^0: V^{\otimes d-2}\otimes V^{\otimes 2}\to V^{\otimes d-2}\otimes V^{\otimes 2}$.
\begin{Pp}[\cite{GW}, Theorem~10.1.6] 
\label{Pp_description_cB_d}
If $d\ge 2$ then $\cB_d$ is generated, as an associative $k$-algebra, by $\rho_d(g), g\in S_d$ and $P$. 
\end{Pp}

\ssec{Proof of Proposition~\ref{Pp_2.2.3}} 1) Write $V_0$ for the standard representation of $\check{G}$. Let $V=V_0\oplus \Qlb$, this is the standard representation of $\check{H}$. Consider the isomorphism $\alpha_V$ of Definition~\ref{Def_1.1.2} for this particular representation of $\check{H}$. So,
$$
\alpha_V: \H^{\la}_G(V_0, K)\oplus (K\boxtimes\Qlb[1])\,\iso\, (K\boxtimes E_0[1])\oplus (K\boxtimes\Qlb[1])
$$
on $\Bun_G\times X$. 

 Consider the component $\alpha_{12}: K\boxtimes\Qlb[1]\to K\boxtimes E_0[1]$ of $\alpha_V$.  For any $m\in\ZZ$, the map $\H^m(\alpha_V): \H^m(K\boxtimes\Qlb[1])\to \H^m(K\boxtimes E_0[1])$ on $\Bun_G\times X$ is zero, because $\H^0(X, E_0)=0$. By Lemma~\ref{Lm_4.0.2}, the components 
$$
\alpha_{V_0}: \H^{\la}_G(V_0, K)\to K\boxtimes E_0[1]
$$ 
and $K\boxtimes\Qlb[1]\to K\boxtimes\Qlb[1]$ of $\alpha_V$ are isomorphisms.
 
 Any map $\H^m(K\boxtimes E_0)\to \H^m(K\boxtimes\Qlb)$ on $\Bun_G\times X$ is zero, because $\H^0(X, E_0^*)=0$. So, for any $m\in\ZZ$ the map $\H^m(\alpha_V)$ is diagonal with respect to the decomposition $V=V_0\oplus\Qlb$. 

 We obtained the isomorphism $\alpha_{V_0}$ for the standard representation $V_0$ of $\check{G}$. By H1) of Definition~\ref{Def_1.1.2}, $\alpha_{V, V}$ is $S_2$-equivariant. Since $V_0\otimes V_0\subset V\otimes V$ is $S_2$-stable, the corresponding map $\alpha_{V_0, V_0}$ is also $S_2$-equivariant.
In view of Section~\ref{Sect_reformulation_GL_n}, we are done. 

\medskip\noindent
2) The composition $\check{G}\toup{\kappa}\check{H}\toup{\pr}\check{G}$ is the identity, where $\pr$ is the projection. For $V\in \Rep(\check{G})$ define $\alpha_V$ as $\alpha_{\Res^{\pr}(V)}$, where $\Res^{\pr}(V)$ denotes the restriction of $V$ via $\pr$. This gives the required $E_{\check{G}}$-Hecke property. 

\medskip\noindent
3) We derive our result from (\cite{G5}, Theorem~10.5.2). Let us work for a moment on the level of $\DG$-categories on prestacks as in \select{loc.cit.} First, we show that $K\in \D^{\prec}(\Bun_G)$ has nilpotent singular support in the sense of \cite{G5}. For any $\lambda\in\Lambda^+_G$, the irreducible $G$-module $V^{\lambda}$ appears in $V^{\otimes d}$, where $d$ is such that $\lambda\in Par(d,n)$. So, $\H^{\la}_G(V^{\lambda}, K)$ is a direct summand of $\H^{\la}_G(V^{\otimes d}, K)\,\iso\, K\boxtimes E^{\otimes d}[1]$. By (\cite{G5}, 10.3.7), $K$ has nilpotent singular support. 

Let $Lift_E$ be the scheme of lifting of $E$ to a $\check{G}$-local system on $X$. Then $Lift_E$ identifies with $\mu_n$. Namely, any symplectic form on $E$ equals $a\omega$ for $a\in\mu_n$. The element $a=1$ corresponds to the $\check{G}$-local system $E_{\check{G}}$.

 In the notations of (\cite{G5}, Theorem~10.5.2), $\QCoh(\LocSys^{restr}_{\check{G}}(X))$ acts on $Shv_{\Nilp}(\Bun_G)$. So, the fibre of $Shv_{\Nilp}(\Bun_G)$ over $E\in\Bun_H$ is a $\DG$-category over $Lift_E$. Since $K$ is equipped with a $E$-Hecke property with respect to $\kappa$, it decomposes canonically as $K\,\iso\, \oplus_{a\in \mu_n} K_a$ in a way compatible with the $E$-Hecke eigen-property with respect to $\kappa$, here $K_a$ is the summand lying in the fibre of $Shv_{\Nilp}(\Bun_G)$ over $(E, a\omega)$.
Thus, each $K_a$ is compatibly equipped with a Hecke property for the local system $(E, a\omega)\in Lift_E$. 

  Note that $(E, a\omega)$ and $(E,\omega)$ are isomorphic as $\check{G}$-local systems on $X$, a choice of $b\in\mu_{2n}$ with $b^2=a$ provides an isomorphism $(E, a\omega)\to (E, \omega)$ given by the multiplication by $b$. Thus, $K_a$ is naturally a $E_{\check{G}}$-Hecke eigen-sheaf. This completes the proof, which is however, not explicit. 
  
   We present below an argument allowing to obtain the above decomposition explicitly as well as the Hecke property of each $K_a$.
Keep notations of Section~\ref{Section_5.3_Weyl_construction}. 
We obtain the isomorphisms (\ref{iso_for_Hecke_key}) using Weyl's construction for $\Sp_{2n}$ recalled in Section~\ref{Section_5.3_Weyl_construction}. 
 
\smallskip 
\noindent 
{\bf Step 1} For $d>0$ consider the isomorphism on $\Bun_G\times X$ 
\begin{equation}
\label{iso_for_Votimes_d_Hecke} 
\alpha_{V^{\otimes d}}:\H^{\la}_G(V^{\otimes d}, K)\,\iso\,  K\boxtimes E^{\otimes d}[1] 
\end{equation}
given by the $E$-Hecke property of $K$ with respect to $\kappa$. Let $S_d$ act on $E^{\otimes d}$ naturally. By H1) and H2) of Definition~\ref{Def_1.1.2}, (\ref{iso_for_Votimes_d_Hecke}) is $S_d$-equivariant. We claim that $\alpha_{V^{\otimes d}}$ is diagonal with respect to the decomposition (\ref{finer_decomposition}). 

 First let $d=2$. One has $\cH(V^{\otimes 2})=(\Sym^2 V)\oplus V^{\lambda}$ with $\lambda=(1,1,0,\ldots,0)$. Besides, $F^0_2=\Qlb$ on which $S_2$ acts by the sign character. 
By the above arguments based on \cite{G5},
$$
\alpha_{\wedge^2 V}: \H^{\la}_G(\wedge^2 V, K)\,\iso\, K\boxtimes (\wedge^2 E)[1]
$$ 
is diagonal with respect to the decomposition 
$\wedge^2 V\,\iso\, V^{\lambda}\oplus V^0$ for $\lambda=(1,1,0,\ldots,0)$. This is the only place where we appeal to \select{loc.cit.} in this algorithmic part of the proof. Thus, $\alpha_{V^{\otimes 2}}$ is diagonal with respect to (\ref{finer_decomposition}). 
 
 For $d\ge 2$ first restrict the isomorphism $\alpha_{V,\ldots, V}$ to a diagonal of codimension one in $X^d$ and use the above diagonal decomposition of $\alpha_{V\otimes V}$ for this diagonal. Then further restrict to the main diagonal in $X^d$. From the description of $\cB_d$ given in Proposition~\ref{Pp_description_cB_d} we conclude that for all $m\in\ZZ$ the map $\H^m(\alpha_{V^{\otimes d}})$ is $\cB_d$-equivariant. Since (\ref{finer_decomposition}) is the isotypic decomposition under the action of $\cB_d$, $\alpha_{V^{\otimes d}}$ is diagonal with respect to the decomposition (\ref{finer_decomposition}). 

 For $\lambda\in\Lambda^+_G$ set $E^{\lambda}=E^{V^{\lambda}}$. We conclude that each component
$$
\alpha_{V^{\lambda}\otimes F^{\lambda}_d}: \H^{\la}_G(V^{\lambda}\otimes F^{\lambda}_d, K)\to K\boxtimes E^{\lambda}\otimes F^{\lambda}_d[1]
$$
of $\alpha_{V^{\otimes d}}$ is a $\cB_d$-equivariant isomorphism. 

 Recall that $\D^{\prec}(S)$ is idempotent complete for an algebraic stack locally of finite type $S$. By Lemma~\ref{Lm_idempotents}, the map $\alpha_{V^{\lambda}\otimes F^{\lambda}_d}$ writes as an isomorphism denoted 
$$
\alpha_{V^{\lambda}, d}: \H^{\la}_G(V^{\lambda}, K)\to K\boxtimes E^{\lambda}[1]
$$
tensored by $F^{\lambda}_d$. Thus, for any  $0\le r\le \frac{d}{2}$ and $\lambda\in Par(d-2r, n)$ we obtained an isomorphism $\alpha_{V^{\lambda}, d}$. 

\smallskip\noindent
{\bf Step 2} Consider the isomorphism $\alpha_{V^0, 2}: K\boxtimes \Qlb\to K\boxtimes\Qlb$ on $\Bun_G\times X$, recall that $V^0$ is the trivial $\check{G}$-module. Consider the representation $V^{\otimes 2n}$ of $\check{H}$, it contains the trivial representation $\det V$ of $\check{H}$. Using H2) of the $E$-Hecke structure of $K$ with respect to $\kappa$ for the representations $V^{\otimes 2}, \ldots, V^{\otimes 2}$ taken $n$ times, we learn that $\alpha_{V^0, 2}\comp\ldots\comp\alpha_{V^0, 2}: K\boxtimes\Qlb\to K\boxtimes\Qlb$ on $\Bun_G\times X$ is the identity, the composition with itself being taken $n$ times. 
 
 By Lemma~\ref{Lm_idempotents}, there is a decomposition $K\,\iso\, \oplus_{a\in \mu_n} K_a$ in $\D^{\prec}(\Bun_G)$ such that $\alpha_{V^0, 2}$ acts on $K_a$ by $a$. This decomposition is preserved by each isomorphism $\alpha_{U}$ for $U\in\Rep(\check{H})$. For each $a\in\mu_n$, we twist the $E$-Hecke property of $K_a$ with respect to $\kappa$ by a suitable element of the center $\mu_{2n}$ of $\check{H}$ as in Remark~\ref{Rem_counterexample}. So, we may and do assume from now on that $\alpha_{V^0, 2}$ is the identity on each $K_a$. 

\smallskip\noindent
{\bf Step 3} For $d>0$ and $\lambda\in Par(d,n)$ set $\alpha_{V^{\lambda}}=\alpha_{V^{\lambda}, d}$. We claim that for any $d\ge 0$, $0\le r\le \frac{d}{2}$, and $\lambda\in Par(d-2r, n)$ one has $\alpha_{V^{\lambda}, d}=\alpha_{V^{\lambda}}$. Indeed, consider the natural embedding
$$
V^{\lambda}\otimes F^{\lambda}_d\otimes (F^0_2)^{\otimes r}\hook{}
V^{\otimes d-2r}\otimes V^{\otimes 2r}\,\iso\, V^{\otimes d}
$$
Applying H2) of the $E$-Hecke property of $K$ with respect to $\kappa$ for the collection $V^{\otimes d-2r}, V^{\otimes 2},\ldots, V^{\otimes 2}\in\Rep(\check{H})$, one gets 
$$
(\alpha_{V^{\lambda}, d-2r})\comp\alpha_{V^0,2}\comp\ldots\comp\alpha_{V^0,2}\mid_{\Bun_G\times\triangle(X)}=\alpha_{V^{\lambda, d}}
$$

\smallskip
\noindent
{\bf Step 4} For any $\cV\in\Rep(\check{G})$ write $\cV=\oplus_{\lambda} (V^{\lambda}\otimes\Hom(V^{\lambda},\cV))$. Set 
$$
\alpha_{\cV}=\mathop{\oplus}_{\lambda\in\Lambda^+_G} \alpha_{V^{\lambda}}\otimes \id_{\Hom(V^{\lambda},\cV)}
$$
We claim that the isomorphisms $\alpha_{\cV}$ provide a $E_{\check{G}}$-Hecke property of $K$. It remains to check Properties H1) and H2) of Definition~\ref{Def_1.1.2}.  

 For any $\cV\in\Rep(\check{G})$ we may pick $\cW\in\Rep(\check{H})$ and a decomposition $\Res^{\kappa}(\cW)\,\iso\, \cV\oplus \cV'$ in $\Rep(\check{G})$ such that the isomorphism $\alpha_{\cW}$ is diagonal with respect to this decomposition. This formally implies H1) of Definition~\ref{Def_1.1.2}. 
 
 Property H2) of the $E$-Hecke eigen-sheaf $K$ with respect to $\kappa$ implies H2) for $K$ as a $E_{\check{G}}$-Hecke eigensheaf. Namely, given $\lambda,\lambda'\in \Lambda_G^+$ let $d,d'\ge 0$ be such that $\lambda\in Par(d,n), \lambda\in Par(d',n)$. Then the restriction of $\alpha_{V^{\lambda}\boxtimes V^{\lambda'}}\mid_{\Bun_G\times \vartriangle(X)}$ is described as the corresponding part of $\alpha_{V^{\otimes d}\boxtimes V^{\otimes d'}}\mid_{\Bun_G\times \vartriangle(X)}$ by the above.

     
\medskip\noindent
4) the proof is similar to 3). The inclusion $\kappa: \check{G}\hook{}\check{H}$ is obtained from $\Gm\times \Sp_4\hook{} \Gm\times\SL_4$, $(x,y)\mapsto (x^{-1}, y)$ by passing to the quotient under the diagonally embedded $(-1, -1)$. 
Let $Lift_{E_{\check{H}}}$ be the scheme of liftings of $E_{\check{H}}$ to a $\check{G}$-local system on $X$. Then $Lift_{E_{\check{H}}}$ identifies with $\mu_2$. Namely, any symplectic form $\wedge^2 E\to\chi$ compatible with a given isomorphism $\det E\,\iso\, \chi^{\otimes 2}$ is $a\omega$ for $a\in\mu_2$. The rest ot the proof is as in 3). 
\QED

\appendix 
\section{Finite field case}
\label{sect_finite_field}

\ssec{} Assume $k_0\subset k$ is a finite subfield, $X$ comes from a curve $X_0$ defined over $k_0$. Assume in the situation of Theorem~\ref{Thm_2.3.8} that $E_{\check{\GG}}$ comes from a $\check{\GG}$-local system $E_{0, \check{\GG}}$ on $X_0$. In this section we check that the function trace of Frobenius of the complex $F_G(K_{E,\chi, \HH})$ is nonzero. 

 Let $\AA$ be the ad\`ele ring of $X$. Recall that D. Soudry has shown in \cite{S1} that irreducible automorphic cuspidal generic representations of $\GG(\AA)$ satisfy the strong multiplicity one property. This is the reason for which we get a particular irreducible automorphic representation of $\GG(\AA)$ attached to $E_{0,\check{\GG}}$. 

The local Langlands conjecture for $\GG$ over a non-archimedian local field of characteristic zero has been established in \cite{GT2}. It has been extended to the case of local non-archimedian field of characteristic $p>2$ in \cite{Gana}. 

 The local theta-correspondence for the dual pair $(\GSp_4, \GSO_6)$ over a local non-archimedian field of characteristic zero and residual characteristic $p>2$ is completely established in (\cite{GT}, Theorem~8.3 and Proposition~13.1). 

\sssec{} The argument below is due to W. T. Gan. The proof is essentially as in (\cite{GT2}, Theorem~12.1(iii)), where a similar claim is established for number fields instead of the function field of $X$. Recall that $\HH\,\iso\, \GL_4\times\Gm/\{(z, z^{-2})\mid z\in\Gm\}$, so an irreducible automorphic representation of $\HH(\AA)$ writes $\Pi\boxtimes\mu$, where $\Pi$ (resp., $\mu$) is a representation of $\GL_4(\AA)$ (resp., $\AA^*$) as in \select{loc.cit}. Let $\Pi\boxtimes\mu$ be the irreducible automorphic cuspidal representation of $\HH(\AA)$ attached to the extension of scalars of $E_{0,\check{\GG}}$ via $\kappa: \check{\GG}\hook{}\check{\HH}$. It suffices to check that the global theta-lift $\Theta(\Pi\boxtimes\mu)$ of $\Pi\boxtimes\mu$ to $\GG(\AA)$ is an irreducible cuspidal globally generic representation attached to $E_{0,\check{\GG}}$. By construction, the partial twisted exterior square $L$-function $L^S(s, \Pi, \wedge\otimes\mu^{-1})$ has a pole at $s=1$. By a result of Jacquet-Shalika \cite{JS}, this is equivalent to $\Pi$ having a nonzero Shalika period with respect to $\mu$. In \cite{S1} and (\cite{GT3}, Proposition 3.1), the first Whittaker coefficient of $\Theta(\Pi\boxtimes\mu)$ is expressed in terms of the Shalika period of $\Pi$ with respect to $\mu$. So, this first Whittaker coefficient is nonzero. The cuspidality of $\Theta(\Pi\boxtimes\mu)$ is proved as in \select{loc.cit}. Thus, $\Theta(\Pi\boxtimes\mu)$ is a globally generic cuspidal representation of $\GG(\AA)$. We are done.

\section{Abelian categories over stacks}
\label{Appendix_B}

In this section we introduce some notions related to \cite{G} and prove Proposition~\ref{Pp_A.2.3} below. 

\ssec{} Let $K$ be an algebraically closed field of characteristic zero. All the stacks (and morphisms of stacks) we consider are defined over $K$. 

 All the stacks we consider in this section are assumed algebraic locally of finite type and such that the diagonal map $\cY\to\cY\times\cY$ is affine. For such a stack $\cY$ one has the notion of a sheaf of abelian categories over $\cY$ (\cite{G}, Section~9). We use the notions and results of \cite{G} freely. Write $\Aff/\cY$ for the category of affine schemes over $\cY$. 

\sssec{} Let $\cC$ be an abelian $K$-linear category, assume $\cC$ presentable in the sense of (\cite{HTT}, Definition 5.5.0.1). 

Let $A$ be a $K$-algebra, assume $\cC$ is a category over $\Spec A$ and $f: \Spec A\to\Spec B$ is a morphism of $K$-schemes. Then $\cC$ can also be viewed as a category over $\Spec B$. This is the operation of direct image of $\cC$ under $f$, write $f_*\cC$ for this category over $\Spec B$. 

\begin{Lm} 
\label{Lm_005}
1) Let $M$ be a $B$-module. If $X\in\cC$ then $M\otimes_B(B\otimes_A X)\,\iso\, M\otimes_A X$ canonically. \\
2) If $B'\gets A'\to A$ is a diagram of $K$-algebras, $B=B'\otimes_{A'} A$, and $\cC$ is a category over $\Spec A$ then $\cC\otimes_{A'} B'\,\iso\, \cC\otimes_A B$ canonically as $B$-linear categories.
\end{Lm}

 More generally, if $f: \cY\to\cY'$ is an affine schematic representable morphism of stacks, and $\cC$ is a sheaf of categories over $\cY$, we define the direct image sheaf $f_*\cC$ as a sheaf of categories over $\cY'$ as follows. If $g': S'\to\cY'$ is an object of $\Aff/\cY'$ and $\bar f: S\to S'$ is obtained from $f$ by the base change under $g'$ then we set $(f_*\cC)_{S'}=\bar f_*\cC$. By Lemma~\ref{Lm_005}, we get indeed a sheaf of categories in the sense of \cite{G}. 
 
\sssec{}  Let $\cC$ be a sheaf of categories over $\cY$, $f: \cY\to\cY'$ is an affine schematic representable morphism of stacks, $g: \cZ'\to \cY'$ a morphism of stacks. Let $\bar f: \cZ\to\cZ'$ be obtained from $f$ by the base change under $g$. Write $\bar g: \cZ\to \cY$ be the projection. Then $g^*(f_*\cC)\,\iso\, \bar f_*(\bar g^*\cC)$ canonically.  

\ssec{} 
\label{Sect_A.2}
From now on the stacks $\cY$ we consider will satisfy the assumpions of (\cite{G}, Section~17), so a sheaf of categories over $\cY$ by (\cite{G}, Theorem~18) is a datum of a category $\cC$ (which we assume $K$-linear abelian presentable), and an action $*: \Vect_{\cY}\times\cC\to\cC$ of $\Vect_{\cY}$ on $\cC$ exact in each variable. Here $\Vect_{\cY}$ is the symmetric monoidal category of vector bundles on $\cY$. 

\sssec{} Let $H$ a connected reductive group over $K$, $G\subset H$ a closed connected reductive subgroup. Write $\Rep(G)$ for the category of finite-dimensional representations of $G$, set $\ov{\Rep}(G)=\Ind\Rep(G)$. Let $\cC$ be a category over $B(G)$, so $\Rep(G)$ acts on $\cC$.  
 
 The category $Hecke(\cC,G)$ of Hecke objects in $\cC$ under the action of $\Rep(G)$ is the category of pairs $(x, \alpha)$, where $x\in\cC$, and $\alpha$ is a collection of isomorphisms $\alpha_V: V\ast x\,\iso\, x\otimes \und{V}$ for $V\in\Rep(G)$ satisfying the compatibility conditions of (\cite{AG}, Section~2.2). Recall that $\cC\times_{B(G)}\Spec K$ identifies canonically with the category of Hecke objects in $\cC$ under the action of $\Rep(G)$ by \select{loc.cit.}
 
 For an algebra $\cA$ in $\ov{\Rep}(G)$ write $\cA-mod^r(\cC)$ for the category of right $\cA$-modules in $\cC$. Consider the space of functions $\cO_G$ as an algebra object of $\Rep(G)$, where $G$ acts on $\cO_G$ by right translations. For $V\in\ov{\Rep}(G)$ write $\und{V}$ for the underlying vector space. The following is well-known, we give a proof to recall the construction.
\begin{Lm} 
\label{Lm_A.2.2}
One has canonically $\cO_G-mod^r(\cC)\,\iso\, Hecke(\cC,G)$.   
\end{Lm}
\begin{proof} 
Let $(x, \alpha)\in \Hecke(\cC, G)$ with $x\in\cA$. One gets the action map $a: x\ast \cO_G\to x$ as the composition 
$$
x\ast \cO_G\toup{\alpha_{\cO_G}} x\otimes \und{\cO_G}\toup{\epsilon} x
$$ 
Here $\epsilon: \cO_G\to K$ is the counit, the restriction to $1\in G$. See also the proof of (\cite{G}, Theorem~18), apply it for the map $\Spec K\to B(G)$. 

 In the other direction, let $(x, a)\in \cO_G-mod^r(\cC)$, where $a: x\ast\cO_G\to x$ is the action map. For $V\in\Rep(G)$ the matrix coefficient gives a map $V\otimes \und{V^*}\to \cO_G$ in $\Rep(G)$. Composing $x\ast (V\otimes \und{V^*})\to x\ast \cO_G\toup{a}x$, by adjointness we get $\alpha_V: x\ast V\to x\otimes\und{V}$.  
\end{proof} 
 
\sssec{}  Let $f:B(G)\to B(H)$ be the natural map. Then $f_*\cC$ is the same category $\cC$ viewed as a category with the action of $\Rep(H)$ via $G\hook{} H$. Note that $G\backslash H\,\iso\, B(G)\times_{B(H)} \Spec K$, so $\cC\times_{B(H)}\Spec K\,\iso\, \cC\times_{B(G)} G\backslash H$ is a category over $G\backslash H$.

Assume $\cC^0$ is an abelian $K$-linear category, in which every object has a finite length, and $\cC\,\iso\, \Ind(\cC^0)$. Since $\cC^0$ admits finite colimits, $\cC$ is presentable by (\cite{HTT}, 5.5.1.1). Assume for any $x\in\cC^0$, $\dim_K\End_{\cC}(x)<\infty$. Assume the action of $\Rep(G)$ on $\cC$ comes (by the functoriality of $\Ind$) from an action of $\Rep(G)$ on $\cC^0$. 

 Write $\cO_{G\backslash H}$ for the space of functions on $G\backslash H$, we view it as an algebra object in $\ov{\Rep}(G)$, where $G$ acts by right translations.

\begin{Pp} 
\label{Pp_A.2.3}
Let $0\ne x\in \cC\times_{B(G)} G\backslash H$ whose image in $\cC$ lies in $\cC^0$. \\
i) There is a closed point $\Spec K\to G\backslash H$ such that $x\otimes_{G\backslash H}\Spec K\in \cC$ is non zero. \\
ii) $x$ admits a finite filtration $0=x_0\subset x_1\subset\ldots\subset x_d=x$ in $\cC\times_{B(G)} G\backslash H$ such that for $1\le i\le d$, $\cO_{G\backslash H}$ acts on $x_i/x_{i-1}$ via some closed point $\cO_{G\backslash H}\to K$ of $G\backslash H$. 
\end{Pp}
\begin{Rem} View $\cO_H$ (resp., $\cO_G$) as an algebra in $\ov{\Rep}(G)$, where $G$ acts by the right translations. We may view $x$ in Proposition~\ref{Pp_A.2.3} as $x\in\cC^0$ together with a structure of a right $\cO_H$-module given by the action map $a: x\ast\cO_H\to x$. Then a closed point $\Spec K\to G\backslash H$ yields by the base change $H\to G\backslash H$ 
a $G$-equivariant morphism $G\to H$, hence a morphism of algebras $\cO_H\toup{\tau}\cO_G$ in $\ov{\Rep}(G)$. 
By definition, $x\otimes_{H/G}\Spec K$ is $x\otimes_{\cO_H} \cO_G\in \cO_G-mod^r(\cC)$. 

 If $\gm=\Ker(\tau)$ then $x\ast \gm\to x\to x\otimes_{\cO_H} \cO_G\to 0$ is exact in $\cC$, so $x\otimes_{H/G}\Spec K$ is a quotient of $x$ in $\cC^0$. 
\end{Rem} 

\sssec{Proof of Proposition~\ref{Pp_A.2.3}} The forgetful functor $\cC\times_{B(G)} G\backslash H\to \cC$ is exact and faithful. So, $x$ is of finite length as an object of $\cC\times_{B(G)} G\backslash H$. If $y\in \cC\times_{B(G)} G\backslash H$ is irreducible such that its image in $\cC$ lies in $\cC^0$ then $\dim_K\End_{\cC\times_{B(G)} G\backslash H}(x)<\infty$. So, the space of functions $\cO_{G\backslash H}$ acts on $y$ via some closed point $\xi: \cO_{G\backslash H}\to K$ of $G\backslash H$. We get $y\otimes_{\cO_{G\backslash H}} K\,\iso\, y$, where the
map $\cO_{G\backslash H}\to K$ is $\xi$. Since for any $\mu: \cO_{G\backslash H}\to K$ the functor $\cC\times_{B(G)} G\backslash H\to \cC\times_{B(G)}\Spec K$, $z\mapsto z\otimes_{\cO_{G\backslash H}} K$ of base change by $\mu$ is right exact, our claim follows. (See also Proposition~\ref{Pp_A.2.9} below). \QED

\sssec{} The rest of Section~\ref{Sect_A.2} is not used in the paper and is added for convenience of the reader. Let $A$ be a $K$-algebra of finite type, $\cD$ is an abelian presentable category over $\Spec A$. Assume $\cD^0$ is an abelian $K$-linear category, in which every object has a finite length, and $\cD\,\iso\, \Ind(\cD^0)$. Assume for any irreducible object $x\in\cD^0$, $\End_{\cD}(x)\,\iso\, K$. 
 
\begin{Lm} 
\label{Lm_1.1.2}
Let $X\in \cD^0$ be irreducible. The $A$-action $A\to\End_{\cD}(X)$ factors through some closed point $A\to K\,\iso\, \End_{\cD}(X)$ of $\Spec A$. \QED
\end{Lm}

The following is an analog of Nakayama's lemma.

\begin{Lm} 
\label{Lm_A.2.9}
Let $X\in\cD^0$. Assume $X\otimes_A K=0$ for any $K$-point $\Spec K\to \Spec A$. Then $X=0$.
\end{Lm}
\begin{proof} 
1) Assume our claim true for any $Y$ irreducible. The functor $\cC\to \cC, X\mapsto X\otimes_A K$ is right exact. If $X\to Y$ is a surjection with $Y$ irreducible then $Y\otimes_A K=0$ for any closed point of $\Spec A$, so $Y=0$. Since $X$ is of finite length, $X=0$. So, it suffices to prove our claim for any $X$ irreducible.

\medskip\noindent
2) Assume $X$ irreducible. By Lemma~\ref{Lm_1.1.2}, $A$ acts on $X$ via some closed point $\Spec K\to \Spec A$. For this point we get $X\otimes_A K\,\iso\, X$. So, $X=0$.
\end{proof}

The following is an immediate consequence of Lemma~\ref{Lm_A.2.9}.
\begin{Pp}
\label{Pp_A.2.9}
If $0\ne X\in\cD^0$ then there is a closed point $\Spec K\to\Spec A$ such that $X\otimes_A K\ne 0$.
\end{Pp} 

\ssec{} 
\label{Sect_A.3}
Here is a kind of application we have in mind. Use notations of Section~\ref{sect_Notation}. Let $G$ be a connected reductive group over $k$, $\check{G}$ its Langlands dual group over $\Qlb$. Pick $x\in X$. Let $\Rep(\pi_1(X,x))$ be the category of finite-dimensional continuous representations of $\pi_1(X,x)$ over $\Qlb$. Let $\cC$ be an abelian presentable $\Qlb$-linear category with commuting actions of $\Rep(\pi_1(X,x))$ and $\Rep(\check{G})$. Both action functors $\cC\times\Rep(\check{G})\to\cC$, $(x, V)\mapsto x\ast V$ and $\cC\times\Rep(\pi_1(X,x))\to\cC$, $(x, W)\mapsto x\ast W$ are assumed exact in each variable. Let $\sigma: \pi_1(X,x)\to \check{G}$ be a continuous homomorphism. For $V\in\Rep(\check{G})$ write $V_{\sigma}$ for the composition $\pi_1(X,x)\toup{\sigma} \check{G}\to \GL(V)$. 

\sssec{} 
\label{Sect_def_Hecke_category_appendix}
One defines the category $\Hecke(\cC,\sigma)$ of $\sigma$-Hecke eigen-sheaves in $\cC$ as the category of pairs $(x,\alpha)$, where $x\in\cC$, $\alpha$ is a collection of isomorphisms $\alpha_V: x\ast V\,\iso\, x\ast V_{\sigma}$ for $V\in\Rep(\check{G})$ satisfying the compatibility conditions as in (\cite{AG}, Section~2.2). Assume that for any $W$ in $\Rep(\pi_1(X,x))$ or in $\Rep(\check{G})$ the functor $\cC\to\cC$, $x\mapsto x\ast W$ is right adjoint to the functor $\cC\to\cC, x\mapsto x\ast W^*$. 
 
 Consider $\cO_{\check{G}}$ as an algebra object in $\Rep(\check{G}\times \pi_1(X,x))$, where $\check{G}$ (resp., $\pi_1(X,x)$) act on $\cO_{\check{G}}$ via left translations (resp., right translations via the homomorphism $\sigma$). Then $\cO_{\check{G}}-mod^r(\cC)\,\iso\, \Hecke(\cC,\sigma)$ as in Lemma~\ref{Lm_A.2.2}.
 
  Another way to spell this is as follows. The category $\cC$ aquires a new action of $\Rep(\check{G})$ as the composition 
$$
\cC\times\Rep(\check{G})\toup{\id\times \Res^{\sigma}}\cC\times\Rep(\pi_1(X,x))\to\cC
$$
We refer to it as the \select{new} action. Let $\Rep(\check{G}\times\check{G})$ act on $\cC$ so that the first factor acts through the old action, and the second one through the new one. Then 
$$
\cO_{\check{G}}-mod^r(\cC)\,\iso\, \cC\times_{B(\check{G}\times\check{G})} B(\check{G}),
$$  
where the map $\check{G}\to \check{G}\times\check{G}$ is the diagonal.
\sssec{} Let $H$ be a connected reductive group over $k$, $\kappa: \check{G}\hook{}\check{H}$ be an inclusion. Let $\Hecke(\cC,\kappa\sigma)$ be the category of pairs $(x,\alpha)$ as for $\Hecke(\cC,\sigma)$, with the difference that $\alpha_V$ is given for $V\in\Rep(\check{H})$ only. It is understood that $\Rep(\check{H})$ acts on $\cC$ via the restriction through $\kappa$ and the old action of $\check{G}$. 

 View $\cC$ as a category over $B(\check{G}\times\check{G})$, hence also over $B(\check{H}\times\check{H})$ via the map $\kappa\times\kappa: \check{G}\times\check{G}\to\check{H}\times\check{H}$. One has naturally 
$$
\Hecke(\cC,\kappa\sigma)\,\iso\, \cC\times_{B(\check{H}\times\check{H})} B(\check{H})\,\iso\, \cC\times_{B(\check{G}\times\check{G})} (\check{G}\times\check{G})\backslash (\check{H}\times\check{H})/\check{H}
$$ 
In the latter formula, $(\check{G}\times\check{G})\backslash (\check{H}\times\check{H})/\check{H}$ is the stack quotient of $\check{H}\times\check{H}$ by $\check{G}\times\check{G}\times\check{H}$, where $\check{G}\times\check{G}$ (resp., $\check{H}$) acts by left (resp., right) translations.
 
 Assume $\cC^0$ is an abelian $\Qlb$-linear category, in which every object has a finite length, and $\cC\,\iso\,\Ind(\cC^0)$. Assume the action of $\pi_1(X,x)\times\check{G}$ on $\cC$ comes by functoriality of $\Ind$ from its action on $\cC^0$.  

There is a relation between $\Hecke(\cC,\sigma)$ and $\Hecke(\cC,\kappa\sigma)$ analogous to Proposition~\ref{Pp_A.2.3}, whose precise formulation is left to a reader. 

\sssec{Example} Assume $K_1,\ldots, K_r$ are irreducible perverse sheaves on $\Bun_G$ such that for any $V\in\Rep(\check{G})$, $\H^{\la}_G(V, K_i)\,\iso\,\oplus_{j=1}^r (E_j^i(V)[1]\boxtimes K_j)$ for some local systems $E_j^i(V)$ on $X$. Let $\Perv(X\times\Bun_G)$ be the category of perverse sheaves on $X\times\Bun_G$. Let $\Rep(\pi_1(X,x))$ act on $\Perv(X\times\Bun_G)$ so that $W\in \Rep(\pi_1(X,x))$ sends $K$ to $\pi_1^*W\otimes K$ for the projection $\pr_1: X\times\Bun_G\to X$. Let $\cC^0\subset \Perv(X\times\Bun_G)$ be the smallest full abelian subcategory containing $\Qlb\boxtimes K_i[1]$ for all $i$, stable under extensions and the action of $\Rep(\pi_1(X,x))$. Then it satisfies all the assumptions of Section~\ref{Sect_A.3}. 

\sssec{} For the derived category $\D(\Bun_G)$ the definition of a Hecke eigen-sheaf as in Section~\ref{Sect_def_Hecke_category_appendix} is not satisfactory as in mentioned in (\cite{BD}, Definition~5.4.2 and Remark after it). This is why we are actually using Definition~\ref{Def_1.1.2} taken from (\cite{G2}, Section~2.8).

\bigskip\noindent
\select{Acknowledgements.} I would like to thank W. T. Gan, who has explained to me the result of the classical theory of automorphic forms from Appendix~\ref{sect_finite_field}. I am also very grateful to Sam Raskin for fruitful discussions and for indicating me a mistake in the first version of this paper. I am grateful to Dennis Gaitsgory for answering my questions and sending me a preliminary version of \cite{G5}. I also thank Vincent Lafforgue for fruitful discussions.

\end{document}